\documentclass{amsart}
\usepackage[T1]{fontenc}
\usepackage{latexsym,amssymb,amsmath}

\usepackage[all,ps]{xy}
\usepackage{amsthm}
\usepackage{longtable}
\usepackage{epsfig}
\usepackage{mathrsfs}
\usepackage{hhline}
\usepackage{epic}
\usepackage{enumerate}
\usepackage{graphicx}
\usepackage{xcolor}

 \newcommand{\sym}{\mathfrak{S}}

 \newcommand{\Tr}{\operatorname{Tr}}

 \newcommand{\Ind}{\operatorname{Ind}}
 \newcommand{\Res}{\operatorname{Res}}

 \newcommand{\cyc}[1]{\langle\,#1\,\rangle}

 \newcommand{\C}{\mathbb{C}}

\newcommand{\N}{\mathbb{N}}
 
 \newcommand{\Z}{\mathbb{Z}}

 \newcommand{\Irr}{\operatorname{Irr}}

\newcommand{\cal}[1]{\mathcal{#1}}

 \newcommand{\dis}{\displaystyle}
 \newcommand{\la}{\lambda}

 \newcommand{\da}{\delta}
  
  \newcommand{\sa}{\sigma}

 \newcommand{\ga}{\gamma}
 \newcommand{\e}{\varepsilon}

  \usepackage{xcolor}
 
  \newcommand{\ula}{\underline{\lambda}}
    \newcommand{\ueta}{\underline{\eta}}
  \newcommand{\uga}{\underline{\gamma}}
  \newcommand{\uw}{\underline{w}}
 \newcommand{\umu}{\underline{\mu}}
 \newcommand{\unu}{\underline{\nu}}
 \newcommand{\uT}{\underline{T}}

  \newcommand{\End}{\operatorname{End}}
  \newcommand{\Mat}{\operatorname{Mat}}

\newtheorem{theorem}{Theorem}[section] % 1st argument is your name for it
\newtheorem{lemma}[theorem]{Lemma}     % 2nd argument is what is printed
\newtheorem{corollary}[theorem]{Corollary}
\newtheorem{proposition}[theorem]{Proposition}

\newtheorem{convention}[theorem]{Convention}

\newtheorem{definition}[theorem]{Definition}
\newtheorem{example}[theorem]{Example}
\theoremstyle{definition}
\newtheorem{remark}[theorem]{Remark}

\title[]% end with percent
{Perfect Isometries Between Blocks of Complex Reflection Groups}
\author{Olivier Brunat}

\address{Universit\'e Paris-Diderot Paris 7\\ Institut de math\'ematiques de
         Jussieu -- Paris Rive Gauche\\ UFR de math\'e\-matiques\\ Case
7012\\ 75205 Paris Cedex 13\\
         France.}
\email{olivier.brunat@imj-prg.fr}

\author{Jean-Baptiste Gramain}

\address{Institute of Mathematics, 
University of Aberdeen, King's College \\
Fraser Noble Building, Aberdeen AB24 3UE, UK
}
\email{jbgramain@abdn.ac.uk}
\thanks{The second author was supported by the EPSRC grant
\emph{Combinatorial Representation Theory} EP/M019292/1}


\begin{document}
\begin{abstract}
in this paper, we prove that, given any integers $d$, $e$, $r$ and $r'$, and a prime $p$ not dividing $de$, any two blocks of the complex reflection groups $G(de,e,r)$ and $G(de,e,r')$ with the same $p$-weight are perfectly isometric.
\end{abstract}
\maketitle

\section{Introduction}
\label{sec:intro}
In the last 30 years, a lot of research in modular representation theory
of finite groups has been fuelled by Brou\'e's Abelian Defect
Conjecture. This predicts that any $p$-block $B$ of a finite group $G$
which has abelian defect group $P$ should be derived equivalent to its
Brauer correspondent $b$ in $N_G(P)$ (see \cite{Broue}). Several
refinements of this conjecture have been formulated, which involve deep
structural correspondences, such as {\emph{splendid equivalences}} or
{\emph{Rickard equivalences}}. At the level of complex irreducible
characters, all of these conjectures predict the existence of a
{\emph{perfect isometry}} between $B$ and $b$.

The first step towards proving Brou\'e's Abelian Defect Conjecture for
the symmetric group was proved by Enguehard in \cite{Enguehard}. He
showed that, if $B$ and $B'$ are $p$-blocks of the symmetric groups
$\sym_m$ and $\sym_n$ respectively, and $B$ and $B'$ have the same
{\emph{$p$-weight}}, then $B$ and $B'$ are perfectly isometric. In this
paper, we generalize Enguehard's result to the infinite family of
complex reflection groups. More precisely, we show that, given any
integers $d$, $e$, $r$ and $r'$, and a prime $p$ not dividing $de$, any
two blocks of the complex reflection groups $G(de,e,r)$ and $G(de,e,r')$
with the same $p$-weight are perfectly isometric (see Theorem
\ref{thm:perfisom}).

The paper is organised as follows. In Section \ref{sec:Gd1r}, we
introduce some combinatorial tools we will need throughout the paper. We
then present the already existing parametrizations, due to James and
Kerber (\S \ref{sec:jameskerber}) and to Marin and Michel (\S
\ref{sec:marinmichel}), of the irreducible representations of the wreath
products $G(d,1,r)$, as well as a new parmetrization which is more
convenient for our purposes (\S \ref{sec:brgr}). In Section
\ref{sec:Gdeer}, we construct the irreducible $G(de,e,r)$-modules (\S
\ref{sec:representations}), and obtain some useful formul{\ae} for the
values of certain characters of $G(de,e,r)$ (\S \ref{sec:characters}).
The results of this part are of independent interest; in particular,
the character table of $G(de,e,r)$ is completely determined (see
Theorem~\ref{th:valcar2}, Theorem~\ref{th:classesGeder} and
Equality~(\ref{eq:chidelta})). Note that we do not follow the same
approach as that of~\cite{read}.

Section \ref{sec:perfisom} is devoted to perfect isometries and our main
result, Theorem \ref{thm:perfisom}. In \S \ref{sec:carN}, \S
\ref{sec:blocks} and \S \ref{sec:bijections}, we describe the
irreducible characters and $p$-blocks of $G(de,e,r)$, as well as
bijections between $p$-blocks with the same $p$-weight. Finally, in \S
\ref{sec:isomparf}, we introduce perfect isometries and prove our main
theorem.

\section{Irreducible representations of $G(d,1,r)$}
\label{sec:Gd1r}
Let $d$ and $r$ be positive integers. Let $\mathcal U_d$ be the group
of complex $d$th roots of unity.
Define $G=G(d,1,r)=\mathcal U_d\wr\sym_r=\mathcal U_d^r\rtimes\sym_r$.
The elements of $G$ are denoted by $(z;\sigma)$, or simply $z\sa$, 
with $z\in\mathcal U_d^r$ and $\sigma \in\sym_r$. In particular,
$\sym_r$ and $\mathcal U_d^r$ are viewed as subgroups of $G$ using the
injections $\sigma\in\sym_r \mapsto (1;\sigma)$ and $z\in\mathcal
U_d^r\mapsto (z;1)$, respectively. 
For any $(z_1,\ldots,z_r)\in \mathcal U_d^r$ and $\sigma\in\sym_r$, 
recall that 
\begin{equation}
\label{eq:conjbase}
\sigma^{-1}(z_1,\ldots,z_r)\sigma=(z_{\sigma(1)},\ldots,z_{\sigma(r)}). 
\end{equation}

Let $\zeta$ be a generator of $\mathcal U_d$.
Write
$t=(\zeta,1,\ldots,1)\in \mathcal U_d^r$ and $s_i=(i\ i+1)\in\sym_r$ for
$1\leq i\leq r-1$. In particular, $G=\langle t,\,s_1,\,\ldots,\,s_{r-1}
\rangle$.

\subsection{Tableaux}

Let $\lambda$ be a partition of $r$ and let $T$ be a tableau of shape
$\lambda$ whose entries are distinct integers. For $(u,v)\in\Z^2$, denote
by $E(T,(u,v))$ the entry of $T$ in the box in row $u$ and column $v$. 
Furthermore, we
write $E(T)$ for the set of integers occurring in $T$. Set
$k=|\lambda|$ and assume that
$E(T)=\{t_1,\,\ldots,\,t_k\}$ with $t_1<\cdots<t_k$. 
Denote by $\operatorname{ST}(\lambda;t_1,\ldots,t_k)$ the set of
standard tableaux of shape $\lambda$ with respect to
$\{t_1,\,\ldots,\,t_k\}$, that is, tableaux $T$ of shape $\lambda$ filled by
the set of integers $\{t_1,\,\ldots,\,t_k\}$ in such a way that the entries in $T$ are
increasing across the rows and the columns of $T$.

Now, for each $T \in \operatorname{ST}(\lambda;t_1,\ldots,t_k)$, define the
tableau $\theta(T)$ of shape $\lambda$ to be such that
$$E(T,(u,v))=t_j\quad\Longleftrightarrow\quad E(\theta(T),(u,v))=j.$$
Write
$\operatorname{ST}(\lambda):=\operatorname{ST}(\lambda;1,\ldots,k)$ for
the set of usual standard tableaux of shape $\lambda$. Then

\begin{lemma}
\label{lem:thetamap}
The map $\theta$ induces a bijection between
$\operatorname{ST}(\lambda;t_1,\ldots,t_k)$ and
$\operatorname{ST}(\lambda)$. 
\end{lemma}

\subsection{Coset representatives for Young subgroups}

Let $r$ be a positive integer. A \emph{composition} of $r$ of length $d$
is a $d$-tuple $(c_0,\ldots,c_{d-1})$ of non-negative integers such that
$\sum_{i=0}^{d-1} c_i=r$. Let $c=(c_0,\ldots,c_{d-1})$ be a composition of $r$.
Write $I_c$ for the set of integers $0\leq i\leq d-1$ such that $c_i\neq 0$.
We set $C_0=0$ and $C_i=c_0+\cdots+c_{i-1}$ for any
$i\in I_c$, and
$E_i=\{C_i+1,\ldots,C_i+c_i\}$. %Note that, if $c_i=0$ for some $0\le
%i\leq k-1$, then $C_{i+1}=C_i$ and $E_i=\emptyset$. 
Now, we can associate to $c$ the \emph{Young subgroup}
$\sym_c=\sym_{E_{i_0}}\times\cdots\times\sym_{E_{i_s}}$ of $\sym_r$, where $I_c=\{i_0, \ldots ,i_s\}$. 
% natural way, where $\sym_{c_i}=\sym_{E_i}\leq\sym_r$. 
Furthermore, for any $i\in I_c$, we denote by
$p_i:\sym_{E_i}\longrightarrow\sym_{c_i}$ the group isomorphism induced by
the bijection $E_i\rightarrow \{1,\ldots,c_i\},\,C_i+j\mapsto j$.
\medskip

Let $E=\{1,\ldots,r\}$.
For any composition $c=(c_0,\ldots,c_{d-1})$ of $r$, define 
\begin{equation}
\label{eq:rep}
\mathcal X_c=\left\{(X_0,\ldots,X_{d-1})\,|\,\bigsqcup_{i=0}^{d-1}
X_i=E,\,|X_i|=c_i\right\}. 
\end{equation}

\begin{lemma}
\label{lem:rep}
Let $c=(c_0,\ldots,c_{d-1})$ be a composition of $r$. For 
$X=(X_0,\ldots,X_{d-1})\in
\mathcal X_c$, % For $i\in I_c$, suppose that $c_i\neq 0$, 
% write 
define $t_X\in\sym_r$ by setting, for all $i\in I_c$ and $1\leq j\leq c_i$,
\begin{equation}
t_X(C_i+j)=x_{i,j}, 
\end{equation}
where $X_i=\{x_{i,1},\ldots,x_{i,c_i}\}$
with $x_{i,1}<\cdots<x_{i,c_i}$.  

% \begin{enumerate}[(i)]
% \item The set $\mathcal T_c=\{t_X\,|\,X\in X_c\}$ is a set of representatives of
% $\sym_r/\sym_c$. 
% \item  
Let $\sigma\in \sym_r$. For $i\in I_c$, define 
$X_i(\sigma)= \{\sigma(x)\,|\,x\in
E_i\}=\{x_{i,1},\ldots,x_{i,c_i}\}$ with $x_{i,1}<\cdots<x_{i,c_i}$, and 
write $X=(X_0(\sigma),\ldots,X_{d-1}(\sigma))\in \mathcal X_c$. Then 
$$\sigma=t_X\widetilde{\sigma}_0\cdots\widetilde{\sigma}_{d-1},$$
where $\widetilde{\sigma}_i\in\sym_{E_i}$ is defined as follows. 
If $i\notin I_c$, then $\widetilde\sigma_i=1$. Otherwise,
% For $x\notin E_i$,
% set $\widetilde\sigma_i(x)=1$. %Otherwise, %$j\in E_i\neq\emptyset$, and
for any $x\in E_i$, 
there 
is a unique $m_x\in\{1,\ldots,c_i\}$ such that $\sigma(x)=x_{i,m_x}$, and we set
$\widetilde\sigma_i(x)=C_i+m_x$. For $x\notin E_i$, we set
$\widetilde{\sigma}_i(x)=x$. 
In particular, $\mathcal T_c=\{t_X\,|\,X\in \mathcal X_c\}$ is a complete set of representatives for
$\sym_r/\sym_c$.
% \end{enumerate}
\end{lemma}

\begin{proof}
Let $1\leq x\leq r$. Since $X\in \mathcal X_c$, there is $i\in I_c$ such that
$x\in X_{i}(\sigma)$, and
$\widetilde\sigma_{i+1}\cdots\widetilde\sigma_{d-1}(x)=x$. Furthermore,
 $\widetilde\sigma_i(x)=C_i+m_x\in X_i(\sigma)$, and it follows that
$\widetilde\sigma_1\cdots\widetilde\sigma_{i-1}(C_i+m_x)=C_i+m_x$.
Finally,
$t_X(C_i+m_x)=x_{i,m_x}=\sigma(x)$, as required. 
By construction, $t_X$ and
$\widetilde\sigma_1,\ldots,\widetilde\sigma_{d-1}$ are uniquely
determined from $\sigma$, hence $\mathcal T_c$ is a complete set of
representatives for $\sym_r/\sym_c$.
\end{proof}

% The elements of $G$ are denoted by $(z;\sigma)$ with
% $z\in(\Z/d\Z)^r$ and $\sigma\in\sym_r$. 
For any composition $c=(c_0,\ldots,c_{d-1})$ of $r$, we set 
$$G_c=\mathcal U_d^r\rtimes \sym_c.$$ 
Write $\pi_c:G\rightarrow \sym_r$ for the natural projection with kernel
$\mathcal U_d^r$. Note that $G_c=\pi^{-1}(\sym_c)$, and that
$\pi(G_c)=\sym_c$, hence $G/G_c $ is in bijection with $\sym_r/\sym_c$. The bijection
is given by $tG_c\mapsto \pi(t)\sym_c$.
Identifying $\sym_r$ to a subgroup of $G$ as above, we can take $\mathcal T_c$ 
for a set of representatives for
$G/G_c$. Furthermore, using Lemma~\ref{lem:rep} and
Relation~(\ref{eq:conjbase}), we deduce that,
if $g=z\sigma\in G$ with $z=(z_1,\ldots,z_r)\in\mathcal U_d^r$ and 
$\sigma\in\sym_r$, then
\begin{equation}
\label{eq:repG}
g=t_X\underbrace{(z_{t_X(1)},\ldots,z_{t_X(r)})
\widetilde\sigma_0\cdots\widetilde\sigma_{d-1}}_{\in
G_c},
\end{equation}
where $X=(X_{0}(\sigma),\ldots,X_{d-1}(\sigma))$ and
$\widetilde\sigma_0,\ldots,\widetilde\sigma_{d-1}$ are as in
Lemma~\ref{lem:rep}.

\subsection{The James-Kerber parametrization}
\label{sec:jameskerber}

For any partition $\lambda$ of $r$, there is a corresponding irreducible
Specht module $V_{\lambda}$ of $\sym_r$. Write
$\psi_{\lambda}:\sym_r\rightarrow \operatorname{GL}(V_{\lambda})$ for
the corresponding irreducible representation of $\sym_r$.
Recall that $V_{\lambda}$ has a $\C$-basis
$v_{\lambda}=\{v_{\lambda,T}\,|\,T\in\operatorname{ST}(\lambda)\}$
such that
\begin{equation}
\label{eq:actionsn}
\psi_{\lambda}(s_i)(v_{\lambda,T})=\frac{1}{a(i,i+1)}
v_{\lambda,T}+(1+\frac{1}{a(i,i+1)})v_{\lambda, T_{i\leftrightarrow
i+1}},
\end{equation}
where $a(i,i+1)$ denotes the distance between the diagonals of
$T$ where $i$ and $i+1$ occur, and $T_{i\leftrightarrow
i+1}$ is the standard tableau of shape $\lambda$ obtained
by exchanging the integers $i$ and $i+1$ in $T$.

Let $\alpha$ denote the identity of $\mathcal U_d$. We can write $\Irr(\mathcal U_d)=\{\alpha^i\,|\,0\leq i\leq d-1\}$.
A $d$-multipartition $\underline
\lambda=(\lambda^{(0)},\,\lambda^{(1)},\ldots,\,\lambda^{(d-1)})$ of $r$ is a $d$-tuple of partitions
such that $\sum_{i=0}^{d-1}|\lambda^{(i)}|=r$. We write this as $\ula \Vdash_d r$, and denote by
$\mathcal{MP}_{r,d}$ the set of $d$-multipartitions of $r$.
Recall that, up to $G$-isomorphism, the
irreducible representations of $G$ are parametrized by
$\mathcal{MP}_{r,d}$ as follows.

For any $\underline
\lambda=(\lambda^{(0)},\,\lambda^{(1)},\ldots,\,\lambda^{(d-1)})$, write
$c_i=|\lambda^{(i)}|$ for $0\leq i\leq d-1$, and
$c=(c_0,\ldots,c_{d-1})$. Now, consider the
irreducible character of $\mathcal U_d^r$
\begin{equation}
\label{eq:chargroupebase}
\alpha_c=\bigotimes_{i=0}^{d-1}\underbrace{\alpha^i\otimes\cdots\otimes\alpha^i}_{c_i\
\text{times}},
\end{equation}
whose inertial subgroup in $G$ is $G_{c}$. Extend $\alpha_c$ to $G_c$ by
setting $\alpha_c(z\sigma)=\alpha_c(z)$ for all $z\in \mathcal U_d^r$ and
$\sigma\in \sym_c$, and denote by $\C w_c$ the corresponding
representation space. Now, for any $i \in I_c$, the space 
$V_{\lambda^{(i)}}$ has a structure of $\sym_{E_i}$-module given by 
the homomorphism $\psi_{\lambda^{(i)}}\circ p_i:\sym_{E_i}\rightarrow
\operatorname{GL}(V_{\lambda^{(i)}})$.
Hence, $V_{\underline\lambda}=V_{\lambda^{(0)}}\otimes\cdots\otimes
V_{\lambda^{(d-1)}}$ is an irreducible $\sym_c$-module, which gives
an irreducible representation of $G_c$ through $\pi_c$. 
Furthermore, to simplify the notation, we identify $\C w_c\otimes
V_{\underline\lambda}$ with $V_{\underline\lambda}$, by setting $z
v=\alpha_c(z)v$ for all $z\in\mathcal U_d^r$ and $v\in
V_{\underline\lambda}$. Now, by Clifford theory, the $G$-module
\begin{equation}
\label{eq:defrepgd1}
W_{\underline \lambda}=\Ind_{G_c}^{G}(V_{\underline\lambda})
\end{equation}
is irreducible, and $\{W_{\underline
\lambda}\,|\,\underline\lambda\in\mathcal{MP}_{r,d}\}$ is a complete set of non-isomorphic irreducible $G$-modules. For any
$\underline \lambda\in\mathcal{MP}_{r,d}$, write
$\vartheta_{\underline\lambda}:G_c\rightarrow
\operatorname{GL(V_{\underline\lambda})}$
and 
$\rho_{\underline\lambda}:G\rightarrow
\operatorname{GL}(W_{\underline\lambda})$
for the corresponding representation of $G_c$ and $G$, respectively.

By definition of the induction representation, the set
$$\mathfrak b_{\underline\lambda}=\left\{t_X\otimes
v_{\lambda^{(0)},T_0}\otimes\cdots\otimes v_{\lambda^{(d-1)},T_{d-1}}\,|\,
X\in \mathcal X_c,\, T_i\in\operatorname{ST}(\lambda^{(i)})\ \text{for }0\leq
i\leq d-1\right \}$$ is a
$\C$-basis of $W_{\underline\lambda}$. Furthermore, for
$z=(z_1,\ldots,z_r)\in\mathcal U_d^r$, $\sigma\in\sym_r$, and 
$t_X\in\mathcal T_c$, there are $t_{X_{\sigma}}\in\mathcal T_c$ and
$\widetilde{\sigma}_0\in\sym_{c_0},\,\ldots,\,\widetilde{\sigma}_{d-1}
\in\sym_{d-1}$
such that 
\begin{equation}
\label{eq:decrepind}
z\sigma
t_X=t_{X_{\sigma}}(z_{t_{X_{\sigma}}(1)},\ldots,z_{t_{X_{\sigma}}(r)})
\widetilde{\sigma}_0\cdots\widetilde{\sigma}_{d-1}
\in\sym_{c}
\end{equation}
(see Relation~(\ref{eq:repG}) applied to $g=z\sigma t_X\in
G$). Therefore, if $v= t_X\otimes
v_{\lambda^{(0)},T_0}\otimes\cdots\otimes v_{\lambda^{(d-1)},T_{d-1}}$, then
\begin{equation}
\label{eq:formuleactioninduite}
\rho_{\underline\lambda}(z\sigma)(v)=
\alpha_c(z_{t_{X_\sigma}(1)},\ldots,z_{t_{X_\sigma}(r)})
t_{X_\sigma}\otimes \widetilde\sigma_0(
v_{\lambda^{(0)},T_0})\otimes\cdots\otimes
\widetilde{\sigma}_{d-1}(v_{\lambda^{(d-1)},T_{d-1}}),
\end{equation}
where
$\widetilde\sigma_i(v_{\lambda^{(i)},T_i})=\psi_{\lambda^{(i)}}\circ p_i (\widetilde\sigma_i)(v_{\lambda^{(i)},T_i})$
for all $0\leq i\leq d-1$.

\subsection{The Marin-Michel parametrization}
\label{sec:marinmichel}

In~\cite[\S2.3]{Marin-Michel}, Marin and Michel give the following model
for $\Irr(G)$. Let
$\underline\lambda=(\lambda^{(0)},\ldots,\lambda^{(d-1)})\in\mathcal{MP}_{r,d}$.
Define $\mathcal T(\underline\lambda)$ to be the set of standard
multi-tableaux of shape $\underline\lambda$, that is, the set of tuples
of tableaux $\underline T=(T_0,\ldots,T_{d-1})$ where 
\begin{itemize}
\item For all $0\leq i\leq d-1$, the tableau $T_i$ is of
shape $\lambda^{(i)}$.
\item The tableaux $T_0,\ldots,T_{d-1}$ are filled by the set of integers
$\{1,\ldots,r\}$ in such a way that each integer appears exactly once in
one of the tableaux, and, for each $i$, the integers appearing in $T_i$ are increasing across the rows and
columns of $T_i$.  
\end{itemize}

Now, the $\C$-vector space $W'_{\underline\lambda}$ with
basis $\mathcal T(\underline\lambda)$ can be given a $G$-module structure
so that
$\{W'_{\underline\lambda}\,|\,\underline\lambda\in\mathcal{MP}_{r,d}\}$ is a
complete set of non-isomorphic irreducible $G$-modules. Write
$\rho'_{\underline\lambda}:G\rightarrow\operatorname{GL}(W'_{\lambda})$
for the corresponding irreducible representation of $G$. 

Denote by $\underline T(1)$ 
the index of the tableau of $\underline T$ containing the integer $1$, and
for $1\leq i\leq r-1$, write $\underline T_{i\leftrightarrow
i+1}\in\mathcal T(\underline\lambda)$ for
the multi-tableau obtained from $\underline T$ by exchanging the integers
$i$ and $i+1$ in $\underline T$.\medskip

With this notation, we have (see \cite[\S2.3]{Marin-Michel}) $\rho'_{\underline\lambda}(t)(\underline
T)=\zeta^{\underline T(1)}\underline T$. Furthermore, for $1\leq i\leq r-1$, 
if $i$ and $i+1$ do not belong to 
the same tableau of $\underline T$, then 
$\rho'_{\underline\lambda}(s_i)(\underline
T)=\underline T_{i\leftrightarrow
i+1}$. Otherwise, 
\begin{equation}
\label{eq:formuleIvanJean}
\rho'_{\underline\lambda}(s_i)\underline T=\frac{1}{a(i,i+1)}\underline T+\left(1+\frac{1}{a(i,i+1)}\right)\underline T_{i\leftrightarrow
i+1}.
\end{equation}
\medskip

\begin{proposition}
\label{prop:isoKerMar}
Let
$\underline\lambda=(\lambda^{(0)},\,\ldots,\lambda^{(d-1)})
\in\mathcal{MP}_{r,d}$. Then 
the linear map $f_{\ula}:W'_{\underline\lambda}\rightarrow
W_{\underline\lambda}$ defined on the basis $\{\underline
T\,|\,T\in\mathcal T(\underline\lambda)\}$ of $W'_{\underline\lambda}$ by
setting, for every $\underline T=(T_0,\,\ldots,\,T_{d-1})\in\mathcal
T(\underline\lambda)$,
$$f_{\ula}(T_0,\ldots,T_{d-1})=t_X\otimes
v_{\lambda^{(0)},\theta(T_0)}\otimes\cdots\otimes
v_{\lambda^{(d-1)},\theta(T_{d-1})},$$
where $X=(E(T_0),\,\ldots,\,E(T_{d-1}))$ and $\theta$ is the map
constructed before Lemma~\ref{lem:thetamap}, is an isomorphism of
$G$-modules.
\end{proposition}

\begin{proof}
First, we remark that $f_{\ula}$ sends a basis to a basis, whence is a bijective
linear map. 
To prove the result, it suffices to show that
$\rho_{\underline\lambda}(g)\circ
f_{\ula}=f_{\ula}\circ\rho'_{\underline\lambda}(g)$ for all
$g\in\{t,s_1,\,\ldots,\, s_{r-1}\}$. 

Let $\underline T=(T_0,\,\ldots,\,T_{d-1})\in \mathcal
T(\underline\lambda)$. Write $c_i=|\lambda^{(i)}|$ for all $0\leq i\leq
d-1$, and set $c=(c_0,\,\ldots,\,c_{d-1})$. Define $t'=
(1,\ldots,1,\zeta,\,1,\ldots,1)$,  where $\zeta$ lies in
position $t_X^{-1}(1)$.
Then Relation~(\ref{eq:repG}) gives
$t t_X=t_X t'$. Furthermore, $1\in
E(T_{\underline{T}(1)})$, thus $t_X^{-1}(1)\in E_{\underline T(1)}$. 
It follows from the linearity of $f_{\ula}$, and from Relations~(\ref{eq:chargroupebase}) and (\ref{eq:formuleactioninduite}),
that
\begin{eqnarray*}
\rho_{\underline\lambda}(t)(f_{\ula}(\underline T))&=& 
\rho_{\underline\lambda}(t)(t_X\otimes
v_{\lambda^{(0)},\theta(T_0)}\otimes\cdots\otimes
v_{\lambda^{(d-1)},\theta(T_{d-1})})\\&=&
\alpha_c(t')t_X\otimes
v_{\lambda^{(0)},\theta(T_0)}\otimes\cdots\otimes
v_{\lambda^{(d-1)},\theta(T_{d-1})}\\
&=&\alpha^{\underline T(1)}(\zeta)t_X\otimes
v_{\lambda^{(0)},\theta(T_0)}\otimes\cdots\otimes
v_{\lambda^{(d-1)},\theta(T_{d-1})}\\
&=&\zeta^{\underline T(1)}f_{\ula}(\underline T)\\
&=&f_{\ula}(\zeta^{\underline T(1)}\underline T)\\
&=&f_{\ula}(\rho'_{\underline\lambda}(t)(\underline T)).\\
\end{eqnarray*}
Now, let $1\leq i\leq r-1$. 
Assume $i$ and $i+1$ do not lie in the
same tableau of $\underline T$, say $i\in E(T_k)$ and $i+1\in E(T_{\ell})$. Then
$s_iT_X=T_{X_{i\leftrightarrow
i+1}}$, where $X_{i\leftrightarrow
i+1}\in \mathcal X_c$ is obtained from $X$ by exchanging $i$ and $i+1$. It
follows from Relation~(\ref{eq:formuleactioninduite}) that
$$ \rho_{\underline\lambda}(s_i)(f_{\ula}(\underline T))= 
t_{X_{i\leftrightarrow
i+1}}\otimes
v_{\lambda^{(0)},\theta(T_0)}\otimes\cdots\otimes
v_{\lambda^{(d-1)},\theta(T_{d-1})}=f_{\ula}(T_{i\leftrightarrow
i+1})=f_{\ula}(\rho'_{\underline\lambda}(s_i)(\underline
T)).$$ 
Assume now that $i$ and $i+1$ lie in the same tableau of $\underline T$,
say $i,\,i+1\in E(T_k)=\{t_1,\ldots,t_m\}$ with $t_1<\cdots< t_m$. Let
$1\leq i'\leq m$ be such that $i=t_{i'}$. Necessarily, we have $t_{i'+1}=i+1$, and 
$s_i t_X= t_X s_{C_k+i'}$.
Thus, Relations~(\ref{eq:formuleactioninduite}) and~(\ref{eq:actionsn}) 
give
\begin{eqnarray*}
\rho_{\underline\lambda}(s_i)(f_{\ula}(\underline T))&=& 
t_X\otimes
v_{\lambda^{(0)},\theta(T_0)}\otimes\cdots\otimes\psi_{\lambda^{(k)}}(s_{i'})v_{\lambda^{(k)},\theta(T_k)}\otimes\cdots\otimes
v_{\lambda^{(d-1)},\theta(T_{d-1})}\\
&=&\frac{1}{a(i',i'+1)}f_{\ula}(\underline
T)+\left(1+\frac{1}{a(i',i'+1)}\right)f_{\ula}(\underline T_{i\leftrightarrow
i+1}).
\end{eqnarray*}
Let $(u,v)$ and $(u',v')$ be such that
$E(T_k,(u,v))=i$ and $E(T_k,(u',v'))=i+1$. Then by construction, we have 
$E(\theta(T_k),(u,v))=i'$ and $E(\theta(T_k),(u',v'))=i'+1$. In
particular, $a(i,i+1)=a(i',i'+1)$ and we deduce from the linearity of $f_{\ula}$
and 
Relation~(\ref{eq:formuleIvanJean}) that
\begin{eqnarray*}
\rho_{\underline\lambda}(s_i)(f_{\ula}(\underline T))&=& 
f_{\ula}\left(\frac{1}{a(i,i+1)}\underline
T+\left(1+\frac{1}{a(i,i+1)}\right)\underline T_{i\leftrightarrow
i+1}\right)\\
&=&f_{\ula}\left(\rho'_{\underline\lambda}(s_i)(\underline T)\right ),
\end{eqnarray*}
as required.
\end{proof}
\subsection{Other descriptions in some special cases}
\label{sec:brgr}

In this section, we assume that there are integers $q$, $r'$ and $d'$
such that $d=qd'$ and $r=qr'$, and we consider multi-partitions
$\underline\lambda\in\mathcal{MP}_{r,d}$ of the form 
\begin{equation}
\label{eq:multisymetrique}
\underline\lambda=(\lambda^{(0)},\ldots,\lambda^{(d'-1)},\lambda^{(0)},\ldots ,\lambda^{(d'-1)},\ldots, \lambda^{(0)},\ldots ,\lambda^{(d'-1)})=
(\underbrace{\underline\mu,\,\ldots,\,\underline\mu}_{q\
\text{times}}),
\end{equation}
where
$\underline\mu=(\lambda^{(0)},\ldots,\lambda^{(d'-1)})\in\mathcal{MP}_{r',d'}$.
Write $c=(c_0,\ldots,c_{d-1})$ and $E_0,\,\ldots,\,E_{d-1}$ as above, 
and $c'=(c_0,\ldots,c_{d'-1})$.

Let $0\leq i\leq q-1$. We set
$L_i=\sym_{E_{i}}\times\cdots\times\sym_{E_{i+d'-1}}$, 
$K_i= \mathcal U_{d}^{r'}\rtimes L_i$, 
$$E'_i=\bigsqcup_{k=0}^{d'-1}E_{i+k},$$
and $H_{i}=\mathcal U_{d}^{r'}\rtimes \sym_{E'_i}$. Note that the character
$\alpha_{r'}^i=\alpha^i\otimes\cdots\otimes\alpha^i\in\Irr(\mathcal
U_{d}^{r'})$ extends to $H_{i}$. Recall that $V_{\underline\mu}$ is an
$L_i$-module. We endow $V_{\underline\mu}$ with a structure of
$K_i$-module where the action of $\mathcal
U_{d}^{r'}\leq K_i$ is given by $\alpha_{r'}^{d'i}\otimes\alpha_{c'}$, and we
denote by $V_{\underline\mu,i}$ the resulting $K_i$-module.
Now set
$W_{\underline\mu,i}=\Ind_{K_i}^{H_i}(V_{\underline\mu,i})$, %=\alpha_{r'}^{d'i}
%\otimes W_{\underline\mu}$, 
and define $X_{\underline\mu,i}$ to be the subset
of elements $Y=(Y_0,\,\cdots,\,Y_{d-1})\in \mathcal X_c$ such that for all
$0\leq j\leq q-1$ and $0\leq k\leq d'-1$, if $j\neq i$, then 
$Y_{jd'+k}=\{jr'+C_k+1,\ldots,jr'+C_{k+1}\}$. In particular, $|Y_{id'+k}|=c_k$
and $\bigsqcup_k Y_{id'+k}=\{ir'+1,\,\ldots,\,(i+1)r'\}$.
Therefore, $\{t_{Y_i}\,|\,Y_i\in X_{\underline\mu,i}\}$ is a system of
coset representatives of
$H_i/K_i$. We also consider the set 
$\mathcal T'$ of tuples
$T'=(T_0,\ldots,T_{d'-1})$ with
$T_j\in\operatorname{ST}(\lambda^{(j)})$ for $0\leq j\leq d'-1$. For
$T'\in\mathcal T'$, write $v_{T'}=v_{\lambda^{(0)},T_0}\otimes\cdots\otimes
v_{\lambda^{(d'-1)},T_{d'-1}}$. In particular, $\{v_{T'}\, | \, T'\in\mathcal
T' \}$ is a basis of $V_{\underline\mu,i}$. Hence, 
$\{t_{Y_i}\otimes v_{T'_i}\,|\, Y_i\in
X_{\underline\mu,i},\,T'_i\in\mathcal T'\}$ is a basis of
$W_{\underline\mu,i}$.
\smallskip 

% 
% $V_{\underline\mu}$ where the action of $\mathcal
% U_{d}^{r}\leq K_i$ is given by $\alpha_{r'}^i\otimes\alpha_{c'}$, then
% viewed as a $G_c$-representation with respect to de direct 

% with
% $r'=|E_i'|$,
%G_{E_{i}}\times\cdots\times G_{E_{i+d'-1}}$,
 
Now, set
\begin{equation}
\label{eq:defH}
H=H_{0}\times\cdots\times H_{q-1}.
\end{equation}
% Note that, for every $0\leq i\leq q-1$, the character
% $\alpha_{r'}^i=\alpha^i\otimes\cdots\otimes\alpha^i\in\Irr(\mathcal
% U_{d}^{r'})$ extends to $H_{i}$. So, we define $W_{\underline\mu,i}=
% \alpha_{r'}^i\otimes W_{\underline\mu}$ and
% 
Consider 
the $H$-module
\begin{equation}
\label{eq:defUmu}
U_{\underline\mu}=W_{\underline\mu}\otimes
W_{\underline\mu,1}\otimes\cdots\otimes W_{\underline\mu,q-1},
\end{equation}
% where
% $W_{\underline\mu,i}=\alpha_{r'}^i\otimes W_{\underline\mu}$. 
% Set
%$c_{\underline\mu}=(c_0,\,\ldots,\,c_{d'-1})$ with 
% $c_j=|\lambda^{(i)}|$
% for $0\leq j\leq d'-1$. 
% and
% $c=(c_{\underline\mu},\ldots,c_{\underline\mu})$ ($q$ times).
and define
$$W''_{\lambda}=\Ind_{H}^G(U_{\underline\mu}).$$
We write
$\rho''_{\underline\lambda}:G\rightarrow
\operatorname{GL}(W''_{\underline\lambda})$
for the corresponding representation of $G$.
\smallskip

\begin{proposition}
\label{prop:isopordre}
The $G$-module $W_{\lambda}''$ has basis
$$\mathfrak b''_{\underline\lambda}=\{t_{X'}\otimes(t_{Y_0}\otimes v_{T'_0})\otimes\cdots\otimes (t_{Y_{q-1}}\otimes
v_{T'_{q-1}})\,|\,X'\in \mathcal X_{(r',\ldots,r')},\,
Y_i\in X_{\underline\mu,i},\,T'_i\in\mathcal T'\}.$$
% 
%  $$t_{X'}t_{Y_0}\cdots
% t_{Y_{q-1}}\otimes v_{\lambda^0,T_0}\otimes\cdots\otimes
% v_{\lambda^{d-1}T_{d-1}}$$ with $Y_i\in X_{\underline\mu,i}$, $X'\in
% X_{(d',d',\ldots,d')}$.
For $X\in \mathcal X_c$, define $l(X)=(X'_0,\,\ldots,\,X'_{q-1})\in
\mathcal X_{(r',\ldots,r')}$, where
$$X'_i=\bigsqcup_{k=0}^{d'-1}X_{id'+k},\quad \text{for all }\ 0\leq i\leq q-1.$$

Write $X'_i=\{x'_{i,1},\,\ldots,\,x'_{i,r'}\}$ with
$x'_{i,1}<\cdots<x'_{i,r'}$, and, for all $0\leq i\leq q-1$ and
$0\leq k\leq d'-1$, consider the element $Y_i(X)\in X_{\underline\mu,i}$ 
such that
$Y_{i,k}(X)=\{id'+j\,|\, x'_{i,j}\in X_{id'+k}\}$.
Then the linear map $f_{\ula}'$ defined by
$$f_{\ula}'(t_{X}\otimes v_{T'_0}\otimes\cdots\otimes
v_{T'_{q-1}})=t_{l(X)}\otimes (t_{Y_0(X)}\otimes
v_{T'_0})\otimes\cdots \otimes (t_{Y_{q-1}(X)}\otimes
v_{T'_{d-1}})$$
is an isomorphism of $G$-modules between $W_{\underline\lambda}$ and
$W''_{\underline\lambda}$.
\end{proposition}

\begin{proof}
Note that
\begin{equation}
\label{eq:Gcproddir}
G_c=K_0\times\cdots\times K_{q-1},
\end{equation}
and, viewed as a $G_c$-representation with respect to the direct 
product~(\ref{eq:Gcproddir}), we have
$$V_{\underline\lambda}=V_{\underline\mu,0}\otimes\cdots\otimes
V_{\underline\mu,q-1}.$$% \quad\text{and}\quad
%\alpha_c=\bigotimes_{i=0}^{q-1}\alpha_{r'}^i\otimes
%\alpha_{c'}.$$ 

By Lemma~\ref{lem:rep}, 
$\{t_{X'}\,|\,X'\in \mathcal X_{(r',\ldots,r')}\}$ is a system of coset
representatives of $G/H$, and 
$\{t_{Y_0}\cdots t_{Y_{q-1}}\,|\,Y_i\in
X_{\underline\mu,i}\}$ is a system of coset representatives for $H/G_c$.
Then there is an isomorphism of $G$-modules
$\kappa_1:\Ind_H^G(\Ind_{G_c}^H(V_{\underline\lambda}))\longrightarrow\Ind_{G_c}^G(V_{\underline\lambda})$
given on any basis $\{v\}$ of $V_{\underline\lambda}$ by $$\kappa_1(t_{X'}\otimes t_{Y_0}\cdots t_{Y_{q-1}}\otimes
v)=t_{X'}t_{Y_0}\cdots t_{Y_{q-1}}\otimes v,$$
where $X'\in \mathcal X_{(r',\ldots,r')}$ and $Y_i\in X_{\underline\mu,i}$.
Furthermore, we have
\begin{eqnarray*}
\Ind_{G_c}^H(V_{\underline\lambda})&=&\Ind_{K_0\times\cdots\times
K_{q-1}}^{H_0\times\cdots\times
H_{q-1}}(V_{\underline\mu,0}\otimes\cdots\otimes
V_{\underline\mu,q-1})\\
&\cong&\Ind_{K_0}^{H_0}(V_{\underline\mu,0})\otimes\cdots\otimes
\Ind_{K_{q-1}}^{H_{q-1}}(V_{\underline\mu,q-1}).
\end{eqnarray*}
The last isomorphism of $H$-modules is for example given by
$$t_{Y_0}\cdots t_{Y_{q-1}}\otimes v_{T'_0}\otimes\cdots
v_{T'_{q-1}}\mapsto (t_{Y_0}\otimes v_{T'_0})\otimes\cdots\otimes (t_{Y_{q-1}}\otimes
v_{T'_{q-1}})$$
for all $Y_i\in X_{\underline\mu,i}$ and $T'_i\in\mathcal T'$.
We thus obtain an isomorphism of $G$-modules
$\kappa_2:\Ind_H^G(\Ind_{G_c}^H(V_{\underline\lambda}))\longrightarrow
\Ind_{H}^G(U_{\underline\mu})$ given by
$$\kappa_2(t_{X'}\otimes t_{Y_0}\cdots t_{Y_{q-1}}\otimes v_{T'_0}\otimes\cdots
v_{T'_{q-1}})=t_{X'}\otimes
(t_{Y_0}\otimes v_{T'_0})\otimes\cdots\otimes (t_{Y_{q-1}}\otimes
v_{T'_{q-1}}).$$
Now, note that the map 
\begin{equation}
\label{eq:bijkappa}
\kappa:\mathcal X_c\rightarrow \mathcal X_{(r',\ldots,r')}\times
X_{\underline\mu,0}\times\cdots\times X_{\underline\mu,q-1},\
X\mapsto(l(X),Y_0(X),\ldots,Y_{q-1}(X))
\end{equation} 
is bijective and that
$$t_X=t_{l(X)}t_{Y_0(X)}\cdots t_{Y_{q-1}(X)}.$$
It follows that $f_{\ula}'=\kappa_2\circ\kappa_1^{-1}$ has the required
property.
\end{proof}

\begin{remark}
\label{rk:isoWi}
Note that $H_0=G(d,1,r')$ and that $W_{\underline\mu,0}$ is the irreducible
representation of $H_0$ labeled by
$(\underline\mu,\emptyset,\ldots,\emptyset)$. In the same way, for every
$0\leq i\leq q-1$, the group $H_i$ can be viewed as a complex reflexion
group $G(d,1,r')$ with support $E_i$. The
irreducible representation $W_{\underline\mu,i}$ of $H_i$ is then
labeled by
$(\emptyset,\ldots,\emptyset,\underline\mu,\emptyset,\ldots,\emptyset)$,
where $\underline\mu$ lies in $i$th-coordinate. In the following, we 
will identify  $H_i$ with $H_0$ as well as
$W_{\underline\mu,i}$ with $W_{\underline\mu,0}$ as follows.
Let $0\leq i\leq q-1$. Write $$\tau_i:\{ir'+1\ldots
(i+1)r'\}\rightarrow \{1,\ldots,r'\},\,ir'+j\mapsto j.$$
Then $\tau_i$ induces a group isomorphism between $H_i$ and $H_0$.
Furthermore, for $Y\in X_{\underline\mu,i}$, define $Y^0\in X_{c'}$ 
by setting $Y^0_k=\tau_i(Y_{id'+k})$ for all $0\leq k\leq d'-1$.
Then the $H_i$-module $W_{\underline\mu,i}$ and the $H_0$-module
$\alpha_{r'}^{d'i}\otimes W_{\underline\mu}$ are isomorphic. An
isomorphism $\mathfrak f_i$ is given on the basis $\{t_Y\otimes v_{T'}\}$ 
by
\begin{equation}
\label{eq:isointer}
t_Y\otimes v_{T'}\mapsto t_{\tau_i(Y)}\otimes v_{T'},
\end{equation}
for all $Y\in X_{\underline\mu,i}$ and $T'\in\mathcal T'$.

\end{remark}

\section{Character formula for the irreducible representations of $G(de,e,r)$}
\label{sec:Gdeer}
Let $e$, $d$ and $r$ be positive integers, and write $G=G(de,1,r)$. 

\noindent
Let
$\underline\varepsilon=(\emptyset,\cdots,\emptyset,(r),\emptyset,\cdots,\emptyset)\in\mathcal{MP}_{r,de}$,
where the non-empty part of $\underline\varepsilon$ lies in position
$de-1-d$.
Then $\varepsilon=\rho_{\underline\varepsilon}$ is a linear character
of $G$ of order $e$, and we denote by $N=G(de,e,r)$ its kernel. In
particular, if $z=(z_1,\ldots,z_r)\in\mathcal U_{de}^r$ and $\sigma\in\sym_r$, then $z\sigma$ lies in $N$ if and only if $\varepsilon(z)=1$, that is $z_1\cdots
z_r\in\mathcal U_d$. 

\subsection{Representations of $G(de,e,r)$}
\label{sec:representations}
Let
$\underline\lambda=(\lambda^{(0)},\ldots,\lambda^{(de-1)})\in
\mathcal{MP}_{r,de}$. Note that, by construction, $\Res_{\mathcal
U_{de}^r}^G(\varepsilon)=\alpha^d\otimes\cdots\otimes\alpha^d\in\Irr(\mathcal
U_{de}^r)$. It follows from Equation~(\ref{eq:defrepgd1}) that
$$\varepsilon\otimes\rho_{\underline\lambda}=\varepsilon\otimes\Ind_{G_c}^G(\vartheta_{\underline\lambda})\cong\Ind_{G_c}^G(\varepsilon\otimes\vartheta_{\underline\lambda})=\Ind_{G_c}^G(\vartheta_{\varepsilon(\underline\lambda)})=\rho_{\varepsilon(\underline\lambda)},$$
where
$\varepsilon(\underline\lambda)=(\lambda^{(d)},\lambda^{(d+1)},\ldots,\lambda^{(de+d)})$;
note that, here, the indices are taken modulo $de$. 

Let $a$ be a divisor of $e$ such that 
$\varepsilon^{a}(\underline\lambda)=\underline\lambda$. Then 
$\lambda^{(da+k)}=\lambda^{(k)}$ for any $k$, and
\begin{equation}
\label{eq:divir}
r=|\langle\varepsilon^{a}\rangle|\sum_{k=0}^{d{a}-1}|\lambda^{(k)}|.
\end{equation}

The set $C_{\underline\lambda}=\{\varepsilon^j\,|\,
\varepsilon^j\otimes\rho_{\underline\lambda}\cong\rho_{\underline\lambda}\}$
is a subgroup of the cyclic group $\langle\varepsilon\rangle$, hence
there is a divisor $b_{\ula}$ of $e$ such that
$C_{\underline\lambda}=\langle\varepsilon^{b_{\ula}}\rangle$. Furthermore, by
Clifford theory, $\Res^G_N(\rho_{\underline\lambda})$ is the sum of
$|C_{\underline\lambda}|$ non isomorphic irreducible $N$-modules.
Following~\cite[\S\,2.4]{Marin-Michel}, they can be described as
follows. 
By Schur's Lemma and the fact that $\C$ is algebraically 
closed, we can choose a bijective linear map
$M_{\ula}\in\operatorname{Hom}_G(\rho_{\underline\lambda},\varepsilon^{b_{\ula}}\otimes\rho_{\underline\lambda})$
such that $M_{\ula}$ has order $|C_{\underline\lambda}|$. 
On the other hand, $M_{\ula}$ is diagonalisable and
has exactly $|C_{\underline\lambda}|$ eigenspaces with eigenvalues in $\mathcal
U_{|C_{\underline\lambda}|}$. Denote by $W_{\underline\lambda,k}$ the eigenspace
attached to the eigenvalue $\zeta^{b_{\ula}dk}$, where ${\cal U}_{de}=\cyc{\zeta}$ (so that ${\cal U}_{|C_{\underline\lambda}|}=\cyc{\zeta^{b_{\ula}d}}$). Then
$\{W_{\underline\lambda,k}\,|\, 0\leq k\leq |C_{\underline\lambda}|-1\}$ 
is the set of
irreducible $N$-modules appearing in the decomposition of
$W_{\underline\lambda}$ into simple $N$-modules. 

For $0\leq k\leq |C_{\underline\lambda}|-1$, 
denote by $\chi_{\underline\lambda,k}$ the character of the $N$-module
$W_{\underline\lambda,k}$ and by 
$$\Delta_{\underline\lambda,k}(g)=\operatorname{Tr}(M_{\ula}^k\circ
\rho_{\underline\lambda}(g))\quad\text{for all } g\in N.$$
Then we have (see~\cite{Marin-Michel})
\begin{equation}
\label{eq:chidelta}
\chi_{\underline\lambda,k}=\frac{1}{|C_{\underline\lambda}|}\sum_{j=0}^{|C_{\underline\lambda}|-1}\zeta^{-db_{\ula}kj}\Delta_{\underline\lambda,j}.
\end{equation}
Now, using the first orthogonality relation, we deduce that, for $0\leq k\leq
|C_{\underline\lambda}|-1$,
\begin{equation}
\label{eq:chardiff}
\Delta_{\underline\lambda,k}=\sum_{j=0}^{|C_{\underline\lambda}|-1}\zeta^{db_{\ula}kj}\chi_{\underline\lambda,j}. 
\end{equation}

%Let $\varepsilon'\in C_{\underline\lambda}$ be of order $b'$. Write $e'=e/b'$. 

\begin{proposition}
\label{prop:actiondeM}
Let $\underline\lambda=(\lambda^{(0)},\ldots,\lambda^{(de-1)})
\in\mathcal{MP}_{r,de}$ and $c=(c_0,\ldots,c_{de-1})$ be such that
$c_i=|\lambda^{(i)}|$ for all $0\leq i\leq de-1$. Let $b_{\ula}$ be a divisor
of $e$ such that
$C_{\underline\lambda}=\langle\varepsilon^{b_{\ula}}\rangle$.
Define $m_{\ula}:\mathcal X_c\rightarrow \mathcal X_c$ by setting, for any
$X=(X_0,\ldots,X_{de-1})\in \mathcal X_c$,
$$m_{\ula}(X)=(X_{db},\ldots,X_{db+de-1}),$$
where indices are taken modulo $de$. Then the linear map
$M_{\ula}\in\operatorname{Hom}_G(\rho_{\underline\lambda},\varepsilon^{b_{\ula}}\otimes\rho_{\underline\lambda})$
as above can be described on the basis $\mathfrak
b_{\underline\lambda}$ of $W_{\underline\lambda}$ as follows. 
$$M_{\ula}(t_X\otimes v_{\lambda^{(0)},T_0}\otimes\cdots\otimes
v_{\lambda^{(de-1)},T_{de-1}})=t_{m_{\ula}(X)}\otimes v_{\lambda^{(db)},T_{db}}
\otimes\cdots\otimes
v_{\lambda^{(db+de-1)},T_{db+de-1}},$$
where
$X\in \mathcal X_c$ and $T_i\in\operatorname{ST}(\lambda^{(i)})$.
\end{proposition}

\begin{proof}
In~\cite[\S\,2.4]{Marin-Michel}, a bijective linear map
$M_{\ula}'\in\operatorname{Hom}_G(\rho'_{\underline\lambda},\varepsilon^{b_{\ula}}\otimes\rho'_{\underline\lambda})$
of order $|C_{\underline\lambda}|$
is described on the basis $\mathcal T(\underline\lambda)$ of
$W'_{\underline\lambda}$ as follows. For every
$T=(T_0,\ldots,T_{de-1})\in\mathcal T(\underline\lambda)$, we set
\begin{equation}
\label{eq:intertwinerIvanJean}
M_{\ula}'(T)=(T_{db},\ldots,T_{db+de-1}).
\end{equation}
Now, using Proposition~\ref{prop:isoKerMar}, we check that
$$M_{\ula}\circ f_{\ula}=f_{\ula}\circ M_{\ula}'.$$
The result follows.
\end{proof}

% For $0\leq k\leq |C_{\underline\lambda}|$, write
% $b'=\operatorname{gcd}(e,bk)$. Then $M^k$ has order $q=e/b'$, $\varepsilon^{b'}(\underline\lambda)=
% \underline\lambda$ and $q$ divides $r$ 
% by Relation~(\ref{eq:divir}). Let $r'\in\N$ be such that $r=qr'$. 
Let $b'$ be a divisor of $|C_{\underline\lambda}|$. Then
$q=\frac{|C_{\underline\lambda}|}{b'}=\frac e{b_{\ula}b'}$ is the order of
$M^{b'}$ and  $\varepsilon^{b_{\ula}b'}(\underline\lambda)=\underline\lambda$.
Hence, Relation~(\ref{eq:divir}) applied to $a=b_{\ula}b'$ gives that $q$ divides
$r$. Let $r'\in\N$ be such that $r=qr'$. 

\begin{proposition}
\label{prop:groupeordre}
We keep the notation as above. Write
$\underline\mu=(\lambda^{(0)},\ldots,\lambda^{(db_{\ula}b'-1)})$ so that
$\underline\lambda=(\underline\mu,\ldots,\underline\mu)$ as in
Relation~(\ref{eq:multisymetrique}).
Define $m_{\ula}'':\mathcal X_{(r',\ldots,r')}\rightarrow \mathcal X_{(r',\ldots,r')}$ by
$$m_{\ula}''(X'_0,\ldots,X'_{q-1})=(X'_1,\ldots,X'_{q-1},\,X'_0),$$
and, for all $\mathfrak u=t_{X'}\otimes(t_{Y_0}\otimes v_{T'_0})\otimes\cdots
\otimes (t_{Y_{q-1}}\otimes v_{T'_{q-1}})$ with $X'\in
\mathcal X_{(r',\ldots,\,r')}$, $Y_i\in X_{\underline\mu,i}$ and $T'_i\in\mathcal
T'$, we set
$$M_{\ula}''(\mathfrak u)=t_{m_{\ula}''(X')}\otimes(t_{Y_1}\otimes v_{T'_0})\otimes\cdots
\otimes (t_{Y_{q-1}}\otimes v_{T'_{q-2}})\otimes (t_{Y_{0}}\otimes
v_{T'_{q-1}}).$$
Then $$M_{\ula}''\circ f_{\ula}'=f_{\ula}'\circ M_{\ula}^{b'}.$$
\end{proposition}

\begin{proof}
We remark that for all $X\in \mathcal X_c$, 
one has $l(m_{\ula}^{b'}(X))=m_{\ula}''(l(X))$, where
$l:\mathcal X_c\rightarrow \mathcal X_{(r',\ldots,r')}$ is the map defined in
Proposition~\ref{prop:isopordre}.
Let $\kappa:\mathcal X_c\rightarrow \mathcal X_{(r',\ldots,r')}\times
X_{\underline\mu,0}\times \cdots \times
X_{\underline\mu,q-1}$ be the bijection defined in
Relation~(\ref{eq:bijkappa}).
Then for all $X\in \mathcal X_c$, 
$$\kappa\circ m_{\ula}^{b'}(X)=(m_{\ula}''(l(X)),Y_1(X),\ldots,Y_{q-1}(X),\,Y_0(X)).$$
The result then follows from Proposition~\ref{prop:isopordre}.
\end{proof}

\subsection{Values of $\Delta_{\underline\lambda,k}$}
\label{sec:characters}

%  Let $r'$ and $q$ be positive integers such that $r=qr'$.
%  Set $G_{r'}=\mathcal U_{de}\wr \sym_{r'}$. For all $0\leq i\leq q-1$,
%  consider the
%  subgroup $H_i$ of $G$ defined before Relation~(\ref{eq:defH}). For $h\in H_i$, denote
%  by
%  $\overline h$ the image of $h$ by the group isomorphism
%  $H_i\rightarrow G_{r'}$ induced by the bijection $\sqcup_k
%  p_{i+k}:E'_i\rightarrow \{1,\,\ldots,\,r'\}$.
%  
\begin{lemma}
\label{lem:valcar1}
Let $\underline\lambda\in\mathcal{MP}_{r,de}$ be such that
$C_{\underline\lambda}=\langle\varepsilon^{b_{\ula}}\rangle$ for some divisor $b_{\ula}$
of $e$. Let $b'$ be a divisor of $|C_{\underline\lambda}|$, 
$q=|C_{\underline\lambda}|/b'$ and $r'\in\N$ be such $r=qr'$. Write
$\mu=(\lambda^{(0)},\ldots,\lambda^{(db_{\ula}b'-1)})$ so that
$\underline\lambda=(\underline\mu,\ldots,\underline\mu)$. For $g\in G$
and $X\in \mathcal X_{(r',\ldots,r')}$, define $X_g\in \mathcal
X_{(r',\ldots,r')}$ and
$g^{X}_{i}\in H_i$ for $0\leq i\leq q-1$ such that
$gt_{X}=t_{X_g}g_0^X\cdots g_{q-1}^X$.
Then
$$\Delta_{\underline\lambda,b'}(g)=\sum_{X\,| \, m_{\ula}''(X_g)=X}\left(
\prod_{i=0}^{q-1}
\alpha_{r'}^{db_{\ula}b'i}(g_i^X)\right)\chi_{\underline\mu}(\overline{g}_0^X
\cdots \overline{g}_{q-1}^X),$$
where $\chi_{\underline\mu}$ denotes the character of the irreducible
representation of
$\Irr(H_0)$ labeled by $(\underline\mu,\emptyset,\ldots,\emptyset)$ and $\overline g_i^X\in H_0$ 
is the image of $g_i^X$ by the isomorphism $H_i\rightarrow H_0$ induced
by the bijection $\tau_i$ given in Remark~\ref{rk:isoWi}.
\end{lemma}

\begin{proof}
By Propositions~\ref{prop:isopordre} and~\ref{prop:groupeordre}, we have
$\rho_{\underline\lambda}(g)=f_{\ula}'^{-1}\circ
\rho''_{\underline\lambda}(g)\circ f_{\ula}'$ and $M_{\ula}^{b'}=f_{\ula}'^{-1}\circ
M_{\ula}''\circ f_{\ula}'$. Hence
$$\Delta_{\underline\lambda,b'}(g)=\operatorname{Tr}(M_{\ula}^{b'}\circ\rho_{\underline\lambda}(g))=\operatorname{Tr}(f_{\ula}'^{-1}\circ
M_{\ula}''\circ\rho''_{\underline\lambda}(g)\circ f_{\ula}')=\operatorname{Tr}(
M_{\ula}''\circ\rho''_{\underline\lambda}(g)).$$
Let $\mathfrak u=t_{X}\otimes(t_{Y_0}\otimes v_{T'_0}) \otimes\cdots\otimes
(t_{Y_{q-1}}\otimes v_{T'_{q-1}})\in\mathfrak b''_{\underline\lambda}$.
We have
$$M_{\ula}''\circ\rho''_{\underline\lambda}(g)(\mathfrak u)=\left(
\prod_{i=0}^{q-1}
\alpha_{r'}^{db_{\ula}b'i}(g_i^X)\right) t_{m_{\ula}''(X_g)}\,
\left(\bigotimes_{i=1}^{q}g_i^X\cdot(t_{Y_i}\otimes
v_{T'_{i-1}})\right),
%\otimes
%g_{0}^X\cdot (t_{Y_{0}}\otimes
%v_{T'_{q-1}}).
$$
where the indices are taken modulo $q$.
To simplify the notation, we denote the basis $\{t_{Y_0}\otimes
v_{T'_0}\,|\, Y_0 \in X_{\underline\mu,0},\,T'_0\in \mathcal T'\}$
of $W_{\underline\mu}$ by
$\mathfrak e=\{e_1,\ldots,\,e_s\}$, and the basis of $W_{\underline\mu,i}$ is then
equal to $\mathfrak f_i^{-1}(\mathfrak e)$ by~(\ref{eq:isointer}). 
Then, by Remark~\ref{rk:isoWi}, for all $0\leq i\leq q-1$,
the matrix of  $g_i^X\cdot (t_{Y_i}\otimes
v_{T_{i-1}'})$ (where the indices are taken modulo $q$) 
with respect to the basis $\mathfrak f_i^{-1}(\mathfrak e)$
is the same as that of $\rho_{\underline\mu}(\overline g_i^X)$ with respect
to the basis $\mathfrak e$. For $h\in H_0$, we denote by
$A_{\underline\mu}(h)=(a_{ij}(h))_{ij}$ 
the matrix of $\rho_{\underline\mu}(h)$ with respect to the basis
$\mathfrak e$. In particular, for $i_0,\ldots,i_{q-1}$,
if we decompose $g_1^X\cdot
e_{i_1}\otimes\cdots\otimes g_{q-1}^X\cdot e_{i_{q-1}}\otimes
g_0^X\cdot e_{i_0}$ with respect to the basis $\{e_{i_0}\otimes\cdots
\otimes e_{i_{q-1}}\}$, then its coefficient in the $e_{i_0}\otimes\cdots\otimes
e_{i_{q-1}}$-coordinate is $$a_{i_0 i_1}(\overline
g_1^X)\cdots a_{i_{q-2} i_{q-1}}(\overline
g_{q-1}^X)\, a_{i_{q-1} i_0}(\overline
g_0^X).$$    

Write $M_{\ula}''\circ\rho''_{\underline\lambda}(g)(\mathfrak u)=\sum_{v\in\mathfrak
b''_{\underline\lambda}} a_v\,v$. Thus, if $a_{\mathfrak u}\neq 0$, then
$m_{\ula}''(X_g)=X$ and, in this case, one has
$$a_{\mathfrak u}= \left(
\prod_{i=0}^{q-1}
\alpha_{r'}^{db_{\ula}b'i}(g_i^X)\right)
a_{i_0 i_1}(\overline
g_1^X)\cdots a_{i_{q-2} i_{q-1}}(\overline
g_{q-1}^X)\, a_{i_{q-1} i_0}(\overline
g_0^X).
$$
Furthermore, note that
$\sum_{i_1,\ldots i_{q-1}}
a_{i_0 i_1}(\overline
g_1^X)\cdots a_{i_{q-2} i_{q-1}}(\overline
g_{q-1}^X)\, a_{i_{q-1} i_0}(\overline
g_0^X)$ is the coefficient $(i_0,i_0)$ of the matrix
$$A_{\underline\mu}(\overline g_1^X)\cdots A_{\underline\mu}(\overline
g_{q-1}^X) A_{\underline\mu}(\overline
g_0^X)=A_{\underline\mu}(\overline g_1^X\cdots \overline
g_{q-1}^X\overline g_0^X),$$
because $\rho_{\underline\mu}$ is a representation of $H_0$. 
It follows that 
\begin{eqnarray*}
\sum_{i_0,\ldots,i_{q-1}}a_{i_0 i_1}(\overline
g_1^X)\cdots a_{i_{q-2} i_{q-1}}(\overline
g_{q-1}^X)\, a_{i_{q-1} i_0}(\overline
g_0^X)&=&\sum_{i_0}a_{i_0 i_0}(\overline g_1^X\cdots \overline
g_{q-1}^X\overline g_0^X)\\
&=&\chi_{\underline\mu}(\overline g_1^X\cdots \overline
g_{q-1}^X\overline g_0^X)\\
&=&\chi_{\underline\mu}(\overline g_0^X\cdots
\overline g_{q-1}^X).
\end{eqnarray*}
The result follows.
\end{proof}
\bigskip

\begin{lemma}
\label{lem:repsatble}
We keep the notation of Lemma~\ref{lem:valcar1}. 
Let $g=(z;\sigma)$ with $z\in\mathcal
U_{de}^r$ and $\sigma\in\sym_r$. Write $\sigma=\sigma_1\cdots\sigma_s$
the cycle decomposition with disjoint support of $\sigma$. Assume that
$\sigma_i$ has length $\ell_i$, %that $\sigma_1=(1\ \cdots\ \ell_1)$ and
and that, for $1\leq j\leq s$, 
\begin{equation}
\label{eq:choixsupport}
\sigma_{j}=(L_j+1\ \cdots \ L_j+\ell_j),
\end{equation}

where $L_1=0$ and $L_j=\ell_1+\cdots +\ell_{j-1}$.
Let $X=(X_0,\ldots,X_{q-1})\in \mathcal X_{(r',\ldots,r')}$, and, for
$0\leq i\leq q-1$, write
$$X_i=\{x_{i,1},\ldots,x_{i,r'}\}\quad\text{with}\
x_{i,1}<\cdots<x_{i,r'}.$$
If
$m_{\ula}''(X_g)=X$, then $\ell_j$ is divisible by $q$ for all $1\leq i\leq
s$. Let $1\leq j\leq s$. Write $l_j=ql'_j$ and $L_j=qL'_j$. Then
there is $0\leq i_0 \leq q-1$ such that $x_{i_0,1}=L_j+1$, and, for all
$1\leq k\leq \ell'_j$ and $0\leq l\leq q-1$, we have
$$x_{i_0-l,L'_j+k}=L_{j}+(k-1)q+l+1,$$
where $i_0-l$ is taken modulo $q$.
\end{lemma}

\begin{proof}
Assume that $m_{\ula}''(X_g)=X$. We have
$X_g=(\sigma(X_0),\ldots,\sigma(X_{q-1}))$. Thus, for all $i\geq
0$, one has
$$\sigma(X_{i+1})=X_i,$$
where the indices are taken modulo $q$. 
Let $1\leq j\leq s$. Assume $L_j+1\in X_{i_0}$ for some $0\leq i_0\leq
q-1$. Now, we prove by induction on $l$ that 
\begin{equation}
\label{eq:permuteX}
L_j+l\in X_{i_0-l+1}.
\end{equation}
Indeed, it is true for $l=1$ and, if we assume it holds for some $l\geq 1$, 
then one has 
$$L_j+l+1=\sigma(L_j+l)\in \sigma(X_{i_0-l+1})=X_{i_0-l},$$
as required. In particular, $L_j+\ell_j\in X_{i_0-\ell_j+1}$. However,
$\sigma(L_j+\ell_j)=L_j+1$, hence
$X_{i_0-\ell_j}=\sigma(X_{i_0-\ell_j+1})=X_{i_0}$ and 
$i_0-\ell_j\equiv i_0\mod q$, that is $q$ divides $\ell_j$. The result
now follows from Relation~(\ref{eq:permuteX}).
\end{proof}
\smallskip

\begin{remark}
\label{rk:reciproquetabnb}
In fact, for all $1\leq j\leq s$, the position of $L_j+1$ completely
determines the integer $x_{i,L'_j+k}$ for $0\leq i\leq q-1$ and $1\leq
k\leq \ell_j'$. Since there are $q$ choices for the place of $L_j+1$, we
deduce that the number of $X\in\mathcal X_{(r',\ldots,r')}$ such that
$m_{\ula}''(X_g)=X$ is $q^{s}$. 
\end{remark}
\smallskip

Recall that the conjugacy classes of $G$ are labeled by
$\mathcal{MP}_{de,r}$ as follows. Let $g=(z;\sigma)\in G$ be with
$z=(z_1,\ldots,z_r)\in\mathcal U_{de}^r$ and $\sigma\in \sym_r$ with
disjoint cycle decomposition $\sigma_1\cdots\sigma_s$. For $1\leq j\leq
s$, write $\widetilde\sigma_j=(z_{(j)};\sigma_j)$ where $z_{{(j)}\,k}=z_k$
if $k$ lies in the support of $\sigma_j$, and $z_{(j)\,k}=1$ otherwise.
The cycle product $\mathfrak c(\widetilde\sigma_j)$ of
$\widetilde\sigma_j$ is then defined to be $\prod_k z_{(j)\,k}$.
Now, we associate to $g$ the multi-partition
$\underline\eta=(\eta_0,\ldots,\eta_{de-1})\in\mathcal{MP}_{r,de}$,
called the \emph{cyclic structure} $\mathfrak c(g)$ of $g$, in such a way
that, for all $1\leq j\leq s$, $\eta_{u}$ has a part of length
$|\sigma_j|$ if and only if $\mathfrak c(\widetilde\sigma_j)=\zeta^u$,
where $\zeta$ is a generator of $\mathcal U_{de}$.
Then two elements $g$ and $g'$ of $G$ are conjugate if and only if
$\mathfrak c(g)=\mathfrak c(g')$.
\smallskip

\begin{convention}
\label{conv:defrep}
Now, for any
$\underline\eta=(\eta_0,\ldots,\,\eta_{de-1})\in\mathcal{MP}_{r,de}$, we
choose as representative for the class of $G$ labeled by
$\underline\eta$ the element $g_{\underline\eta}=(z;\sigma)$, where the
cycles of $g_{\underline\eta}$ are as in~({\ref{eq:choixsupport}}), and,
if $\sigma_j=(L_j+1\cdots L_j+\ell_j)$ is a cycle of $\sigma$ such that
$\mathfrak{c}(\widetilde{\sigma}_j)=\zeta^u$, then $z_{(j)\, k}=1$ if
$k\neq L_j+1$, and $z_{(j)\, L_j+1}=\zeta^u$. 
\end{convention}
\smallskip

%\begin{notation}
For $r\in\N$, we denote by $\mathcal P_r$ the set of partitions of $r$, and we let $\mathcal P= \bigcup_{r \in \N} \cal P_r$. For any $\pi=(\pi_1,\ldots,\pi_t)\in\mathcal P$ and any positive integer $q$, we let
\begin{equation}
\label{eq:multiplepart}
q\star\pi=(q\pi_1,\ldots,q\pi_t)\in\mathcal P.
\end{equation}
%\end{notation}
Furthermore, we write $\ell(\pi)=t$, and, for
$\underline\eta=(\eta_0,\ldots,\eta_{de-1})\in\mathcal{MP}_{r,de}$, we
set $\ell(\underline\eta)=\sum\ell(\eta_u)$. Note that, if
$g=(z;\sigma)$ with $\sigma=\sigma_1 \cdots\sigma_s$ has cyclic
structure $\underline\eta$, then $s=\ell(\underline\eta)$.
\smallskip

\begin{theorem}
\label{th:valcar2}
We keep the notation as in Lemma~\ref{lem:valcar1}. Let $\ueta=(\eta_0,\ldots,\eta_{de-1})\in\mathcal{MP}_{r,de}$ and
$g_{\underline\eta}=(z;\sigma)$ with $z=(z_1,\ldots,z_r)\in\mathcal
U_{de}^r$ and $\sigma=\sigma_1\cdots\sigma_s\in\sym_r$ be as in 
Convention~\ref{conv:defrep}. 
For any $1\leq j\leq s$, write $\xi_j=\mathfrak c(\widetilde\sigma_j)$.
\begin{enumerate}[(i)]
\item If there is $1\leq u\leq de-1$ such that %$q$ does not divide
$\eta_u\notin q\star\mathcal P$, or if there is $1\leq j\leq s$ such that
$\xi_j\notin\mathcal U_{db_{\ula}b'}$ then,
$\Delta_{\underline\lambda,b'}(g_{\underline\eta})=0$.
% \item Assume $\Delta_{\underline\lambda,b'}(g_{\underline\eta})\neq 0$. 
% If $\eta_u\neq
% \emptyset$, then $db'$ divide $u$.
\item Assume $\eta_u=q\star\eta'_u$ for all $0\leq u\leq de-1$ and
$\xi_j\in\mathcal U_{db_{\ula}b'}$ for all $1\leq j\leq s$.
% $\eta_u=\emptyset$ if $db'$ does not divide $u$. 
% Denote by $\mathcal N_{\underline\lambda,\underline\eta}=\{1\leq j\leq
% s\,|\, \xi_j^{d|C_{\underline\lambda}|/q}=1\}$.
Then
$$\Delta_{\underline\lambda,b'}(g_{\underline\eta})=q^{\ell(\underline\eta)}\chi_{\underline\mu}(g'_{\underline\eta}),$$
where $g_{\underline\eta}'\in H_0$ has cyclic structure $(\eta'_0,\ldots,\,\eta'_{de-1})$ and $\umu \in {\cal MP}_{r',de/q}$ is as in Lemma \ref{lem:valcar1}.  
\end{enumerate}
\end{theorem}

\begin{proof}
If there is $1\leq u\leq de-1$ such that $q$ does not divide $|\eta_u|$,
then $\sigma$ has a cycle of length not divisible by $q$. By
Lemma~\ref{lem:repsatble}, there are no $X\in\mathcal
X_{(r',\ldots,r')}$ such that $m_{\ula}''(X_{g_{\underline\eta}})=X$, thus
$\Delta_{\underline\lambda,b'}(g_{\underline\eta})=0$ by
Lemma~\ref{lem:thetamap}, proving the first part of (i).

Denote by $\mathfrak X$ the set of $X\in \mathcal
X_{(r',\ldots,r')}$ such that $m_{\ula}''(X_{g_{\underline\eta}})=X$.
Set $Q=\{0,\ldots,q-1\}$ and, for $\underline i=(i_1,\ldots,i_s)\in Q^s$,
define $\mathfrak X_{\underline i}$ to be the set of
$(X_0,\ldots,X_{q-1})\in\mathfrak X$
such that, for all $1\leq j\leq s$, the integer $L_j+1$ lies in $X_{i_j}$.

Let $X\in\mathfrak X$. Then there is a unique 
$\underline i\in Q^s$ such that $X\in\mathfrak X_{\underline i}$. For $1\leq j\leq s$, write $q\ell'_j$ 
for  the length of $\sigma_j$. Then 
$\sigma_j t_X=t_{X_{\sigma_j}}\sigma'_j$, where 
\begin{equation}
\label{eq:defsigmaprime}
\sigma'_j= (C_{i_j}+L'_j+\ell'_j\ C_{i_j}+L'_j +\ell'_j-1\ \cdots\
C_{i_j}+L'_j+1).
\end{equation}
Furthermore, Relation~(\ref{eq:decrepind}) gives 
$z_{(j)}t_{X_{\sigma_j}}=t_{X_{\sigma_j}}z'$, where
$z'=(z'_1,\ldots,z'_k)\in \mathcal U_{de}^r$ is such that
$z'_{C_{i_j}+L'_j+1}=z_{(j)\,
L_j+1}$ and $z'_{k}=1$ otherwise. 
So, if we set $\widetilde{\sigma}_k^{'X}=1$ if $k\neq i_j$ and 
$\widetilde\sigma_{i_j}^{'X}=z''\sigma_j'$, where
$z''=(z''_1,\ldots,z''_{r'})\in\mathcal U_{de}^{r'}$ is such that
$z''_{L'_j+1}=z'_{C_{i_j}+L'_j+1}$ and $z''_k=1$ otherwise, then
\begin{equation}
\label{eq:deccycle}
\widetilde\sigma_j
t_X=t_{X_{\sigma_j}}\widetilde\sigma_0^{'X}\cdots\widetilde\sigma_{q-1}^{'X},
\end{equation}
where $\widetilde\sigma_k^{'X}\in H_k$. Note that $\mathfrak
c(\widetilde\sigma_{i_j}^{'X})=\mathfrak c(\widetilde\sigma_j)$.
Now, Lemma~\ref{lem:repsatble}
implies that $t_{X_g}=t_{X_{\sigma_1}}\cdots t_{X_{\sigma_s}}$, so that
applying Relation~(\ref{eq:deccycle}) iteratively
to the cycles of $g_{\underline\eta}$, we obtain
$$g_{\underline\eta} t_X=t_{X_g}g_0^X\cdots g_{q-1}^X,$$
where $$g_{k}^X=\prod_{\{1\leq j\leq s\,
|\,i_j=k\}}\widetilde\sigma_{j}^{'X}\in H_k\quad\text{for all}\ 0\leq
k\leq q-1.$$ 
By Relation~(\ref{eq:defsigmaprime}), the cycles
$\widetilde\sigma_j^{'X}$ have disjoint support, and
$$\overline{\sigma}'_j=(L'_j+\ell'_j\ \cdots L'_j+1).
%\qquad \mbox{\color{red} Est-ce que c'est pas plut\^ot $\overline{\widetilde{\sa}_j'^X}$? D\' esol\' e...}
$$
Hence, the element $g'_{\underline\eta}=\overline g_0^X\cdots\overline
g_{q-1}^X\in H_0$ has cyclic structure
$\underline\eta'=(\eta_0',\ldots,\,\eta_{de-1}')$ and does not
dependent on $X\in\mathfrak X$. Therefore, Lemma~\ref{lem:valcar1}
implies
\begin{equation}
\label{eq:inter1}
\Delta_{\underline\lambda,b'}(g_{\underline\eta})=\chi_{\underline\mu}(g'_{\underline\eta})
\sum_{X\in\mathfrak X}\prod_{i=0}^{q-1}\alpha_{r'}^{db_{\ula}b'i}(g_i^X).
\end{equation}

Let $1\leq j\leq s$ and take $X\in \mathfrak X$. 
For $0\leq k\leq q-1$, write $h_k^X\in H_k$ such
that $\widetilde\sigma_j^{-1}g
t_X=t_{X_{\widetilde\sigma_j^{-1}g}h_0^X\cdots h_{q-1}^X}$. 
We denote by $\mathfrak X_{i}$ the set of $X\in\mathfrak
X_{(r',\ldots,r')}$ such that $L_j+1\in X_{i}$. 
If $X\in\mathfrak X_i$, then $g_k^X=h_k^X$ for $k\neq i$ and 
$g_i^X=\widetilde\sigma_i^{'X}h_{i}^X$, 
by Lemma~\ref{lem:repsatble} and Relation~(\ref{eq:defsigmaprime}), and
it follows that
%
%$g_k^X=\widetilde\sigma_{k}^{'X}h_k^X$ for some $h_k^X\in H_k$. Then
%for $X\in\mathfrak X_{\underline i}$ with
%$\underline i=(i_1,\ldots,i_s)\in Q^s$, we have $g_k^X=h_k^X$ for any 
%$k\neq i_j$ and 
%$g_{i_j}^X=\widetilde\sigma_{i_j}^{'X}h_{i_j}^X$, and 
$$\prod_{k=0}^{q-1}\alpha_{r'}^{db_{\ula}b'k}(g_k^X)=\alpha_{r'}^{db_{\ula}b'i}(\widetilde\sigma_{j}^{'X})\prod_{k=0}^{q-1}\alpha_{r'}^{db_{\ula}b'k}(h_k^X)
=\xi_j^{db_{\ula}b'i}\prod_{k=0}^{q-1}\alpha_{r'}^{db_{\ula}b'k}(h_k^X).$$
Note that $J=\sum_{X \in\mathfrak
X_i}\prod_{k=0}^{q-1}\alpha_{r'}^{db_{\ula}b'k}(h_k^{X})$ does not depend on
$i\in Q$. Hence, we obtain
\begin{eqnarray*}
\sum_{X\in\mathfrak
X}\prod_{k=0}^{q-1}\alpha_{r'}^{db_{\ula}b'k}(g_k^X)&=&
\sum_{i_j=0}^{q-1}\sum_{X\in \mathfrak
X_{i_j}}\prod_{k=0}^{q-1}\alpha_{r'}^{db_{\ula}b'k}(g_k^X)\\
&=&\sum_{i_j=0}^{q-1}\sum_{X\in \mathfrak
X_{i_j}}\xi_j^{db_{\ula}b'i_j}\prod_{k=0}^{q-1}\alpha_{r'}^{db_{\ula}b'k}(h_k^X)\\
&=&\sum_{i_j=0}^{q-1}\left(\xi_j^{db_{\ula}b'i_j}\sum_{X \in\mathfrak X_{i_j}}
\prod_{k=0}^{q-1}\alpha_{r'}^{db_{\ula}b'i}(h_k^{X})\right)\\
&=&J \sum_{i_j=0}^{q-1}\xi_j^{db_{\ula}b'i_j}.
\end{eqnarray*}
Now, if $\xi_j\notin\mathcal U_{db_{\ula}b'}$, then %$\sum_{i_j}\xi_j^{dbb'i_j}=q$.
%Otherwise, 
$$\sum_{i_j=0}^{q-1}\xi_j^{db_{\ula}b'i_j}=\frac{\xi_j^{db_{\ula}b'q}-1}{\xi_j^{db_{\ula}b'}-1}=0,$$
since $db_{\ula}b'q=de$ and $\xi_j\in\mathcal U_{de}$. This concludes the proof of (i). If, on the other hand,
%Furthermore, we remark that $db'q=d|C_{\underline\lambda}|$. 
%By iterating the computation to each cycles of $g$, the result follows.
$\xi_j\in\mathcal U_{db_{\ula}b'}$ for all $1\leq j\leq s$, then
$\alpha_{r'}^{db_{\ula}b'i}(g_i^X)=1$ for all $i$. Equation~(\ref{eq:inter1}) now gives 
$\Delta_{\underline\lambda,b'}(g_{\underline\eta})=
\chi_{\underline\mu}(g'_{\underline\eta})|\mathfrak X|$, and (ii) follows from
Remark~\ref{rk:reciproquetabnb} since $s=\ell(\underline\eta)$.
\end{proof}
\smallskip

\begin{remark}
\label{rk:valeurs}
Let
$\underline\lambda=(\underline\mu,\ldots,\underline\mu)\in {\cal MP}_{r,de}$, where $\umu \in {\cal MP}_{r',de/q}$ is as in Lemma \ref{lem:valcar1}. Let $\ueta=(\eta_0,\ldots,\eta_{de-1})\in\mathcal{MP}_{r,de}$ and $g_{\underline\eta}=(z;\sigma)$ be as in
Convention~\ref{conv:defrep}. Set $\sigma=\sigma_1\cdots\sigma_s$ and
$\xi_j=\mathfrak c(\widetilde\sigma_j)$. Assume 
that $q$ divides $\eta_j$ or all $1\leq
j\leq de-1$ and that $\xi_j\in\mathcal U_{de/q}$ for all $1\leq
j\leq s$. 
Since $\xi_j\in\mathcal U_{de/q}$ for all $1\leq j\leq s$, we deduce
that $\eta_u\neq\emptyset$ only if
$q$ divides $u$.
Let
$g'_{\underline\eta}$ be the element of $G(de,1,r')$ with cyclic
structure $(\eta_0',\ldots,\eta_{de-1}')$ described in Convention \ref{conv:defrep}, where $\eta_j=q\eta_j'$, and let
$g_{\underline\eta}^{(q)}\in G(de/q,1,r')$ be the element with cyclic
structure $(\eta'_0,\eta'_q,\eta'_{2q},\ldots)$ described in Convention \ref{conv:defrep}. 

Denote by $\widetilde{\chi}_{\underline\mu}$ the irreducible
character of $G(de/q,1,r')$ labeled by $\underline\mu$. 
% Let $g_{\underline\nu}$ be
%such that $\Delta_{\underline\lambda,b'}(g_{\underline\nu})\neq 0$.
%Write $\underline\nu=(\nu_0,\ldots,\nu_{de-1})$. Since $\xi_j\in\mathcal
%U_{de/q}$ for all $1\leq j\leq s$, we deduce that $\nu_u\neq \emptyset$
%only if $q$ divides $u$.  Consider $g''\in H$ with cyclic structure
%$(\nu'_0,\nu'_q,\nu'_{2q},\ldots)$. 
Following \S\ref{sec:jameskerber}, note that the representation space of
$H_0$ labeled by $(\underline\mu,\emptyset,\ldots,\emptyset)$ and that of $H$ labeled by $\underline\mu$ have the same basis $\mathfrak b$. 
Furthermore,
using the fact
%that $\alpha|_{\mathcal U_{de/q}}$ is the identity, we
that $\Irr(\mathcal U_{de/q})=\{\alpha^i\downarrow_{\, \mathcal U_{de/q}}\mid 0\leq
i\leq r'\}$, we deduce from~(\ref{eq:decrepind})
and~(\ref{eq:formuleactioninduite}) that the actions of
$g'_{\underline\eta}$ and $g_{\underline\eta}^{(q)}$
on $\mathfrak b$ are the same. In particular, we have
$\chi_{\underline\mu}(g'_{\underline\eta})=\widetilde{\chi}_{\underline\mu}(g_{\underline\eta}^{(q)})$ and
\begin{equation}
\label{eq:valeurs}
\Delta_{\underline\lambda,b'}(g_{\underline\eta})=q^{s}\widetilde{\chi}_{\underline\mu}(g_{\underline\eta}^{(q)}). 
\end{equation}
\end{remark}
\smallskip

\begin{example}
\label{exemple:g616}
Consider $G=G(6,1,6)$ and $N=G(6,3,6)$. Let
$\underline\eta=(\emptyset,\emptyset,\emptyset,(6),\emptyset,\emptyset)$.
Then the representative for the conjugacy class of $G$ labeled by
$\underline\eta$ described in Convention~\ref{conv:defrep} is
$g_{\underline\eta}=(\zeta^3,1,1,1,1,1;\sigma)$, where $\sigma=(1\
2\ 3\ 4\ 5\ 6)$. Let $(\lambda_0,\lambda_1)\in\mathcal{MP}_{2,2}$.
Write
$\underline\lambda=(\lambda_0,\lambda_1,\lambda_0,\lambda_1,\lambda_0,\lambda_1)\in\mathcal{MP}_{6,6}$.
We will compute $\Delta_{\underline\lambda,1}(g_{\underline\eta})$. To
this end, consider the
set $\mathfrak X$ as in the
proof of Theorem~\ref{th:valcar2}. Then Lemma~\ref{lem:repsatble} gives
$\mathfrak
X=\{X_1,\,X_2,\,X_3\}$ where
$X_1=(\{1,4\},\,\{3,6\},\,\{2,5\})$,
$X_2=(\{2,5\},\,\{1,4\},\,\{3,6\})$
 and 
$X_3=(\{3,6\},\,\{2,5\},\,\{1,4\})$.
Furthermore,
\begin{align*}
 g_{\underline\eta}t_{X_1}&=t_{X_1}g_0^{X_1}g_1^{X_1}g_2^{X_1}=t_{X_1}\left((\zeta^3,1,1,1,1,1;(1\ 2)),1,1\right)\\
 g_{\underline\eta}t_{X_2}&=t_{X_2}g_0^{X_2}g_1^{X_2}g_2^{X_2}=t_{X_2}\left(1,(1,1,\zeta^3,1,1,1;(3\ 4)),1\right)\\
 g_{\underline\eta}t_{X_3}&=t_{X_3}g_0^{X_3}g_1^{X_3}g_2^{X_3}=t_{X_3}\left(1,1,(1,1,1,1,\zeta^3,1;(5\
6))\right).
\end{align*}
Note that $$\overline g_0^{X_1}\overline g_1^{X_1}\overline g_2^{X_1}= 
\overline g_0^{X_2}\overline g_1^{X_2}\overline g_2^{X_2}=
\overline g_0^{X_3}\overline g_1^{X_3}\overline
g_2^{X_3}=(\zeta^3,1,1,1,1,1;(1\
2))\in H_0,$$
and can be identifed with the element $(\zeta^3,1;(1\ 2))\in G(2,1,2)$.
We have
$$\Delta_{\underline\lambda,1}(g_{\underline\mu})=3\widetilde{\chi}_{(\lambda_0,\lambda_1)}(\zeta^3,1;(1\
2)),$$
where $\widetilde{\chi}_{(\lambda_0,\lambda_1)}$ is the irreducible
character of $G(2,1,2)$ labeled by $(\lambda_0,\lambda_1)$.
\end{example}
\smallskip

\begin{theorem}
\label{thm:memeordre}
With the notation as above, if $0\leq k\leq |C_{\underline\lambda}|-1$,
then
$\Delta_{\underline\lambda,k}(g_{\ueta})=\Delta_{\underline\lambda,b'}(g_{\ueta})$,
where $1\leq b'\leq |C_{\underline\lambda}|$ is such that
the order of $M_{\ula}^k$ is $|C_{\underline\lambda}|/b'$. 
\end{theorem}

\begin{proof} Write $q=|C_{\underline\lambda}|/b'$.
First, we remark that the matrix $M_{\ula}^k$ has order
$q$ if and only if $M_{\ula}^k$ is a generator of the
cyclic group $\langle M_{\ula}^{b'}\rangle$. In particular, there is an integer $1\leq
t\leq q$ coprime to $q$
such that $M_{\ula}^k=M_{\ula}^{b't}$. Now, using 
Proposition~\ref{prop:groupeordre}, we deduce that $M_{\ula}''^s\circ
f_{\ula}'=f_{\ula}'\circ M_{\ula}^k$. Let $g=(z;\sigma)\in G(de,1,r)$ be such that
$\sigma=\sigma_1\cdots\sigma_s$. 
Write $\mathcal Y$ for the set of $Y\in\mathfrak
X_{r',\ldots,r'}$ such that $m_{\ula}''^{t}(Y_g)=Y$.
Now, if $gt_Y=t_{Y_g}g_0^Y\cdots g_{q-1}^Y$, then 
we derive from the proof of Lemma~\ref{lem:valcar1}
that 
\begin{equation}
\label{eq:memeordre}
\Delta_{\underline\lambda,k}(g)=\sum_{Y\in\mathcal Y}\left(
\prod_{j=0}^{q-1}
\alpha_{r'}^{db_{\ula}b'j}(g_j^Y)\right)\chi_{\underline\mu}(\overline{g}_s^Y
\cdots \overline{g}_{s+q-1}^Y),
\end{equation}
where the indices are taken modulo $q$. Furthermore, using the fact that
$\chi_{\underline\mu}$ is a trace, we obtain
$\chi_{\underline\mu}(\overline{g}_s^Y
\cdots
\overline{g}_{s+q-1}^Y)=\chi_{\underline\mu}(\overline{g}_0^Y
\cdots \overline{g}_{q-1}^Y)$. 

Define $f:\mathfrak X\rightarrow \mathcal
Y,\,(X_0,\ldots,X_{q-1})\mapsto (X_0,X_t,\ldots,X_{(q-1)t})$ where the
indices are taken modulo $q$. The map $f$ is well defined because $t$ is
coprime to $q$, whence $j\mapsto jt$ is a bijection of $\Z/q\Z$, and, if
$X\in\mathfrak X$, then, for all $0\leq j\leq q-1$,
$$m_{\ula}''^t(\sigma(X_{jt}))=\sigma(X_{(j+1)t})=\sigma(f(X_{j+1}))=f(\sigma(X_{j+1}))=f(X_j)=X_{jt}.$$
Furthermore, $f$ is bijective since $t$ is coprime to $q$.
Let $X\in \mathfrak X$ and $1\leq j\leq s$. Write $Y=f(X)$. Then $L_j+1\in X_{i_j}$ if
and only if $L_j+1\in Y_{i_jt}$. In particular,
$\{\widetilde{\sigma}_0'^Y,\ldots,\widetilde{\sigma}_{q-1}'^Y\}$ is a
permutation of
$\{\widetilde{\sigma}_0'^X,\ldots,\widetilde{\sigma}_{q-1}'^X\}$ (for
the notation, we refer to the proof of Theorem~\ref{th:valcar2}). Hence,
$\overline g_0^X\cdots\overline g_{q-1}^X$ and $\overline
g_0^Y\cdots\overline g_{q-1}^Y$ have the same cyclic structure, that does
not depend on $X$ and $Y$. We can now conclude as in the end of the proof of
Theorem~\ref{th:valcar2}.
%  Note also that the two sets $\{X\mid
%  m''^s(X_{g_{\underline\eta}})=X\}$ and $\{X\mid
%  m''(X_{g_{\underline\eta}})=X\}$ are equal. Indeed, $m''$ acts on $X$ by
%  permuting circularly the $q$-tuples of coordinates, and since $s$
%  is prime to $q$, we have $\langle m''\rangle=\langle m''^{s}\rangle$ and
%  there is an integer $s'$ such that $m''=m''^{ss'}$.  
%  Comparing Equation~(\ref{eq:memeordre}) and the formula of
%  Theorem~\ref{lem:valcar1}, the result follows.
\end{proof}

\begin{example}
We continue with Example~\ref{exemple:g616}.
We will now compute $\Delta_{\underline\lambda,2}(g_{\underline\eta})$.
We consider $\mathcal Y$ and $f:\mathfrak
X\rightarrow\mathcal Y$ as in the
proof of Theorem~\ref{thm:memeordre}. 

Then
$\mathcal
Y=\{Y_1,\,Y_2,\,Y_3\}$ where
$Y_1=f(X_1)=(\{1,4\},\,\{2,5\},\,\{3,6\})$,
$Y_2=f(X_2)=(\{2,5\},\,\{3,6\},\,\{1,4\})$
 and 
$Y_3=f(X_3)=(\{3,6\},\,\{1,4\},\,\{2,5\})$.
We have
\begin{align*}
 g_{\underline\eta}t_{Y_1}&=t_{Y_1}g_0^{Y_1}g_1^{Y_1}g_2^{Y_1}=t_{Y_1}\left((\zeta^3,1,1,1,1,1;(1\ 2)),1,1\right)\\
 g_{\underline\eta}t_{Y_2}&=t_{Y_2}g_0^{Y_2}g_1^{Y_2}g_2^{Y_2}=t_{Y_2}\left(1,1,(1,1,1,1,\zeta^3,1;(5\
6))\right)\\
 g_{\underline\eta}t_{Y_3}&=t_{Y_3}g_0^{Y_3}g_1^{Y_3}g_2^{Y_3}=t_{Y_3}\left(1,(1,1,\zeta^3,1,1,1;(3\
4)),1\right).
\end{align*}
We again have, for all $1\leq i,\,j\leq 3$
$$\overline g_0^{Y_i}\overline
g_1^{Y_i}\overline g_2^{Y_i}= 
\overline g_0^{X_j}\overline g_1^{X_j}\overline g_2^{X_j}=
(\zeta^3,1,1,1,1,1;(1\
2))\in H_0,$$
and
$$\Delta_{\underline\lambda,1}(g_{\underline\mu})=3\widetilde{\chi}_{(\lambda_0,\lambda_1)}(\zeta^3,1;(1\
2)),$$
as required.
\end{example}
\smallskip

\begin{proposition}
\label{prop:conjvaleurs}
Let $0\leq k\leq |C_{\underline\lambda}|-1$. For any $g\in G(de,1,r)$ and
$x\in G(de,e,r)$, we
have
$$\Delta_{\ula,k}({}^g
x)=\e^{kb_{\underline\lambda}}(g)\Delta_{\ula,k}(x).$$
\end{proposition}

\begin{proof}
Recall that $M_{\ula} \rho_{\ula} M_{\ula}^{-1}=\e^{b_{\ula}} \otimes
\rho_{\ula}$, so that $M_{\ula}^k \rho_{\ula}
M_{\ula}^{-k}=\e^{b_{\ula}k} \otimes \rho_{\ula}$. Thus
\begin{eqnarray*}
\Delta_{\ula,k}({}^g x) & = & \Tr (M_{\ula}^k \rho_{\ula} ({}^g x))
\\
 & = &  \Tr (M_{\ula}^{k} \rho_{\ula} (g) \rho_{\ula} (x) \rho_{\ula} (g)^{-1})
 \\
 & = &  \Tr (\e^{kb_{\underline\lambda}} \otimes \rho_{\ula}(g) M_{\ula}^k \rho_{\ula} (x) \rho_{\ula} (g)^{-1})
\\
& = & \Tr (\e^{kb_{\underline\lambda}}(g) \rho_{\ula}(g) M_{\ula}^k
\rho_{\ula} (x) \rho_{\ula} (g)^{-1})
\\
& = & \e^{kb_{\underline\lambda}}(g) \Tr ( \rho_{\ula}(g) M_{\ula}^k
\rho_{\ula} (x) \rho_{\ula} (g)^{-1})
\\
& = & \e^{kb_{\underline\lambda}}(g) \Tr (M_{\ula}^k \rho_{\ula} (g_{\ueta})),
\end{eqnarray*}
whence
\begin{equation}
\Delta_{\ula,i}(x)=\e^{kb_{\underline\lambda}}(g)\Delta_{\ula,k}(x).
\end{equation}
\end{proof}
Let $\mathfrak g\in G(de,1,r)$ be an element of order $e$ such that
$$G(de,1,r)=G(de,e,r)\rtimes\langle \mathfrak g\rangle.$$ 
Suppose that $\varepsilon(\mathfrak g)=\omega=\zeta^d$. 
Assume that $e$ divides $r$.
For any divisor $q$ of $e$,  define
$$\mathcal
P_{r,ed,q}=\{(\eta_0,\emptyset,\ldots,\emptyset,\eta_q,\emptyset,\ldots,\emptyset,
\eta_{2q},\emptyset,\ldots,\emptyset,\eta_{de-q},\emptyset,\ldots,\emptyset)\mid
\eta_j\in q\mathcal P_{r/q}\}.$$
Furthermore, for any $0\leq j\leq q-1$ and $\ueta\in\mathcal P_{r,ed,q}$, write
$$g_{\ueta,j}={}^{\mathfrak g^{j}}g_{\ueta}.$$

\begin{theorem}
\label{th:classesGeder}
The set $$\bigsqcup_{q|e}\{g_{\ueta,j}\mid \ueta\in\mathcal P_{r,de,q},\
0\leq j\leq q-1 \}$$
is a system of representatives for the conjugacy classes of $G(de,e,r)$.
\end{theorem}

\begin{proof} Write $\mathcal E$ for a system of representatives of the $\langle
\varepsilon\rangle$-orbits of $\mathcal P_{r,de}$.
By Clifford theory from $G(de,1,r)$ to $G(de,e,r)$, the elements of
$$\mathfrak P=\{(\ula,k)\mid \ula\in\mathcal E,\, 0\leq k\leq
|C_{\ula}|-1\}$$
label $\Irr(G(de,e,r))$. For any divisor $q$ of $e$, write $\mathfrak
P_q=\{(\ula,k)\in\mathfrak P\mid |C_{\ula}|=q\}$. In particular,
$$\mathfrak P=\bigsqcup_{q|e}\mathfrak P_q.$$
Note that $(\ula,k)\in \mathfrak P_q$ if and only if
$\ula=(\underline\mu,\ldots,\underline\mu)$, where
$\underline\mu\in\mathcal P_{r/q,de/q}$ is
repeated $q$ times. 
Let $q$ be a divisor of $e$. For any
$\underline\mu=(\mu_0,\ldots,\mu_{de/q-1})\in\mathcal P_{r/q,de/q}$, we define
$f_1(\underline\mu)=(\underline\mu,\ldots,\underline\mu)\in \mathcal
P_{r,de}$ and
$f_2(\underline\mu)=(\lambda_0,\ldots,\lambda_{de-1})\in\mathcal
P_{r,de,q}$
where $\lambda_{qj}=q\mu_j$ for $0\leq j\leq de/q-1$ and
$\lambda_u=\emptyset$ otherwise. The maps $f_1$ and $f_2$ are bijective.
Let $\underline\eta_1$ and $\underline\eta_2$ be two distinct elements of $\mathcal
P_{r,de,q}$. Then for all $0\leq j_1,\,j_2\leq q-1$, the elements
$g_{\underline\eta_1,j_1}$ and $g_{\underline\eta_2,j_2}$ are not
conjugate in $G(de,e,r)$ since they are not conjugate in $G(de,1,r)$.
Let $\underline\eta\in\mathcal P_{r,de,q}$. Write
$\underline\eta'=f_2^{-1}(\eta)$. There exists a character
$\widetilde{\chi}_{\underline\mu}$ of
$G(de/q,1,r/q)$ such that
$\widetilde{\chi}_{\underline\mu}(g_{\underline\eta'})\neq 0$ (we can take for
example $\underline\mu$ such that $\widetilde{\chi}_{\underline\mu}$ is
the trivial character of $G(de/q,1,r/q)$). Then by
Proposition~\ref{prop:conjvaleurs},
%Remark~\ref{rk:valeurs} and Theorem~\ref{th:valcar2}, 
for $0\leq j\leq q-1$, we have
\begin{align*}
 \Delta_{f_1(\underline\mu),1}(g_{\underline\eta,j})&=\varepsilon^{b_{\underline\lambda}}(\mathfrak
g^{j}) \Delta_{f_1(\underline\mu),1}(g_{\underline\eta})\\
&=\omega^{b_{\underline\lambda}j}\Delta_{f_1(\underline\mu),1}(g_{\underline\eta})\\
&=(\omega^{e/q})^j\Delta_{f_1(\underline\mu),1}(g_{\underline\eta}).
\end{align*}
Furthermore, $\Delta_{f_1(\underline\mu),1}(g_{\underline\eta})\neq 0$ by Remark~\ref{rk:valeurs} and Theorem~\ref{th:valcar2}, and
$\Delta_{f_1(\underline\mu),1}(g_{\underline\eta,j_1})\neq
\Delta_{f_1(\underline\mu),1}(g_{\underline\eta,j_2})$ for all $j_1\neq
j_2$ since $\omega^{e/q}$ is a primitive $q$th-root of unity. 
Now, using that $\Delta_{f_1(\underline\mu),1}$ is a class function
of $G(de,e,r)$, we conclude that the elements $g_{\underline\eta,j}$ for $0\leq
j\leq q-1$ are not conjugate in $G(de,e,r)$. Finally, the result
follows from the fact that 
$f_1\circ f_2^{-1}$ induces a bijection between
the sets $\{g_{\underline\eta,j}\mid \underline\eta\in\mathcal
P_{r,de,q},\,0\leq j\leq q-1\}$ and $\{\chi_{\underline\lambda,j}\mid
(\underline\lambda,j)\in\mathfrak P_q\}$.
\end{proof}
\smallskip
\begin{example}
Let $e$ be a prime number and $r$ be a positive integer. By
Theorem~\ref{th:classesGeder}, the elements $g_{\ueta,j}$ where
$\ueta=(\eta,\emptyset,\ldots,\emptyset)$ with $\eta\in\mathcal P_r$ and
$0\leq j\leq e-1$ form a system of representatives for the split
classes of $G(e,e,er)$. For $\lambda\in\mathcal P_{r}$, set
$\underline\lambda=(\lambda,\lambda,\ldots,\lambda)\in\mathcal
P_{er,e}$. By Theorem~\ref{th:valcar2}, Remark~\ref{rk:valeurs},
Theorem~\ref{thm:memeordre} and Proposition~\ref{prop:conjvaleurs}, for
$1\leq k\leq e-1$,
we have
$$\Delta_{\underline\lambda,k}(g_{\ueta,j})=\zeta^{kj}e^{\ell(\eta)}\chi_{\lambda}(\eta),$$
where $\chi_{\lambda}(\eta)$ is the value of the irreducible character
of $\sym_r$ labeled by $\lambda$ on a element with cyclic structure
$\eta$. Now, using Equality~(\ref{eq:chidelta}), we obtain
$$\chi_{\underline\lambda,k}(g_{\underline\eta,j})=
\left\{
\begin{array}{ll}
\dis \frac 1
e\left(\chi_{\underline\lambda}(g_{\underline\eta}))-e^{\ell(\eta)}\chi_{\lambda}(\eta)\right)&
\text{if }k\neq j,
\\
\\
\dis \frac 1
e\left(\chi_{\underline\lambda}(g_{\underline\eta}))+(e-1)e^{\ell(\eta)}\chi_{\lambda}(\eta)\right)&
\text{if }k = j.
\end{array}
\right.$$
In particular, for $e=2$, we recover with our method the result
of~\cite[Thm.\,5.1]{Pfeiffer}.
\end{example}

\section{Perfect Isometries}
\label{sec:perfisom}

Throughout this section, we consider $G=G(de,1,r)$ and its normal subgroup $N=G(de,e,r)$, and we use the notation of Section \ref{sec:representations}.

\subsection{Characters of $N$}
\label{sec:carN}
In order to describe $\Irr(N)$, we will apply Clifford theory from $G$ to $N$. We therefore consider the orbits of $\Irr(G)$ under the action of $G/N \cong \cyc{\e}$, or, equivalently, the $\cyc{\e}$-orbits of the parametrizing set ${\cal{MP}}_{r,de}$. For any $\ula \in {\cal{MP}}_{r,de} $, we denote by $[\ula]$ the $\cyc{\e}$-orbit of $\ula$. Hence $\umu \in [\ula]$ if and only if there exists $s \in \N$ such that $\umu=\e^s(\ula)$. In particular, for any $\umu \in [\ula]$, we have $b_{\ula}=b_{\umu}$ and $|C_{\ula}|=|C_{\umu}|$. Furthermore, we see that $| [\ula] |=\dis \frac{e}{|C_{\ula}|}=b_{\ula}$.

\begin{lemma}
\label{lem:carorbites}
If $\ula, \umu \in {\cal{MP}}_{r,de}$ are such that $[\ula]=[\umu]$, then, with the notation of Section \ref{sec:representations}, $\chi_{\ula,i}=\chi_{\umu,i}$ for all $0 \leq i < |C_{\ula}|$.
\end{lemma}

\begin{proof}

We have $[\ula]=[\umu]$, so there exists $s \in \N$ such that $\rho_{\ula}=\e^s \otimes \rho_{\umu}$, and $W_{\ula}=W_{\umu}=W$. In particular, there exist endomorphisms $M_{\ula}$ and $M_{\umu}$ of $W$ such that $M_{\ula} \rho_{\ula} = \e_{\ula} \rho_{\ula} M_{\ula}$ and $M_{\umu} \rho_{\umu} = \e_{\umu} \rho_{\umu} M_{\umu}$, where, furthermore, $\e_{\ula}=\e^{b_{\ula}}=\e^{b_{\umu}}=\e_{\umu}$.

We therefore have
$$\e^s M_{\ula}\rho_{\umu}=M_{\ula} \e^s \rho_{\umu} = M_{\ula} \rho_{\ula}=\e_{\ula} \rho_{\ula} M_{\ula}=\e^s \e_{\ula} \rho_{\umu} M_{\ula},$$
so that $M_{\ula}\rho_{\umu}=\e_{\ula} \rho_{\umu} M_{\ula}=\e_{\umu} \rho_{\umu} M_{\ula}$ (since $\e_{\ula}=\e_{\umu}$).

Now, since $\e_{\umu} \rho_{\umu} M_{\umu}=M_{\umu} \rho_{\umu}$, we also have $M_{\umu}^{-1} \e_{\umu} \rho_{\umu}=\rho_{\umu} M_{\umu}^{-1}$, so that $M_{\ula}\rho_{\umu}=\e_{\umu} \rho_{\umu} M_{\ula}$ yields $M_{\umu}^{-1} M_{\ula} \rho_{\umu} = M_{\umu}^{-1}\e_{\umu} \rho_{\umu} M_{\ula}=\rho_{\umu} M_{\umu}^{-1} M_{\ula}$. Hence $M_{\umu}^{-1} M_{\ula} \in \End_G(\rho_{\umu})$, and Schur's Lemma shows that $M_{\ula}=\xi M_{\umu}$ for some $\xi \in \C$.

We will show that $\xi=1$. First note that, since $M_{\ula}$ and $M_{\umu}$ both have order $|C_{\ula}|=|C_{\umu}|$, we must have $\xi^{|C_{\ula}|}=1$. Now fix any order on the elements of the bases $\mathfrak{b}_{\ula}$ and $\mathfrak{b}_{\umu}$ of $W$. By Proposition \ref{prop:actiondeM}, we have, for any $t_X \otimes v_{\ula, \uT} \in \mathfrak{b}_{\ula}$,
$$M_{\ula}(t_X \otimes v_{\ula, \uT})=t_{\e^{b_{\ula}}(X)} \otimes v_{\e^{b_{\ula}}(\ula), \e^{b_{\ula}}(\uT)}=t_{\e^{b_{\ula}}(X)} \otimes v_{\ula, \e^{b_{\ula}}(\uT)} \in \mathfrak{b}_{\ula}$$
(since, by definition, $\e^{b_{\ula}}(\ula)=\ula)$. Hence $\Mat(M_{\ula},\mathfrak{b}_{\ula})$ is a permutation matrix. Similarly, $\Mat(M_{\umu},\mathfrak{b}_{\umu})$ is a permutation matrix.

Now, since $\umu=\e^s(\ula)$, there is a bijection $\sigma \colon \mathfrak{b}_{\ula} \longrightarrow \mathfrak{b}_{\umu}$, given by $\sigma (t_X \otimes v_{\ula, \uT})=t_{\e^s(X)} \otimes v_{\umu, \e^s(\uT)})$ for all $t_X \otimes v_{\ula, \uT} \in \mathfrak{b}_{\ula}$. The corresponding change of basis matrix $P_{\sigma}$ from $\mathfrak{b}_{\ula}$ to $\mathfrak{b}_{\umu}$ is therefore also a permutation matrix.

By construction, we have
$$\Mat(M_{\ula},\mathfrak{b}_{\ula})=P_{\sigma}^{-1}\Mat(M_{\ula},\mathfrak{b}_{\umu})P_{\sigma} = \xi P_{\sigma}^{-1}\Mat(M_{\umu},\mathfrak{b}_{\umu})P_{\sigma}$$
(since $M_{\ula}=\xi M_{\umu}$). Since all of these matrices have entries in $\N$, we deduce that $\xi=1$, and thus that $M_{\ula}= M_{\umu}$.

In particular, with the notation of Section \ref{sec:representations}, the eigenspaces $W_{\ula,i}$ and $W_{\umu,i}$ coincide for all $0 \leq i < | C_{\ula} |$, and $\chi_{\ula,i}=\chi_{\umu,i}$ for all $0 \leq i < |C_{\ula}|$.

\end{proof}

\begin{corollary}
\label{cor:cardifforbites}
If $\ula, \umu \in {\cal{MP}}_{r,de}$ are such that $[\ula]=[\umu]$, then $\Delta_{\ula,i}=\Delta_{\umu,i}$ for all $0 \leq i < |C_{\ula}|$.

\end{corollary}

\begin{proof}
This follows immediately from Lemma \ref{lem:carorbites} and (\ref{eq:chardiff}) (since $b_{\ula}=b_{\umu}$ and $|C_{\ula}|=|C_{\umu}|$).

\end{proof}

\begin{remark}
\label{rem:signe}
Suppose $\umu \in {\cal{MP}}_{r,de}$ is such that $|C_{\umu}|$ is even, and take $\da \in \{ \pm 1 \}$. Then $M_{\umu}$ and $\da M_{\umu}$ have the same eigenspaces, and the same set ${\cal U}_{|C_{\umu}|}=\cyc{\omega}$ of eigenvalues, where $\omega=\zeta^{b_{\umu}d}$. For any $\omega^j \in {\cal U}_{|C_{\umu}|}$, we set $$\chi_{\umu,j,\da M_{\umu}}=\Tr(\rho_{\umu} \, | \, E_{\omega^j} ),$$where $E_{\omega^j}$ is the eigenspace of $\da M_{\umu}$ corresponding to the eigenvalue $\omega^j$. We also set, for any $0 \leq i < |C_{\umu}|$,$$\Delta_{\umu,i,\da M_{\umu}}=\Tr((\da M_{\umu})^i\rho_{\umu} \, | \, W_{\umu} ).$$
In particular, we have $\Delta_{\umu,i,\da M_{\umu}}=\da^i \Delta_{\umu,i}$. 

We also have, as in Section \ref{sec:representations},
$$\Delta_{\umu,i,\da M_{\umu}}= \dis \sum_{j=0}^{|C_{\umu}|-1} \omega^{ij} \chi_{\umu,j,\da M_{\umu}} \quad
\mbox{and} \quad \chi_{\umu,j,\da M_{\umu}}= \dis \frac{1}{|C_{\umu}|} \sum_{j=0}^{|C_{\umu}|-1} \omega^{-ij}\Delta_{\umu,i,\da M_{\umu}}.$$
\end{remark}

\subsection{Blocks of $G$ and $N$}
\label{sec:blocks}

We now take any prime $p$ not dividing $de$. The $p$-blocks of $G$ can be described as follows (see \cite[Theorem 1]{OsimaII}). Two irreducible characters $\widetilde{\chi}_{\umu}$ and $\widetilde{\chi}_{\unu}$ of $G$, corresponding to $\umu=(\mu^{(0)},\ldots,\mu^{(de-1)})$ and $\unu=(\nu^{(0)},\ldots,\nu^{(de-1)})$ in ${\cal{MP}}_{r,de}$ lie in the same
$p$-block $B$ of $G$ if and only if, for every $0\leq i\leq de-1$, the partitions
$\mu^{(i)}$ and $\nu^{(i)}$ have the same $p$-core $(\mu^{(i)})_{(p)}=(\nu^{(i)})_{(p)}=\gamma^{(i)}$
and same $p$-weight $w_p(\mu^{(i)})=w_p(\nu^{(i)})=w_i$. The
$de$-tuple $\uw=(w_0,\ldots,w_{de-1})$ (respectively
$\uga=(\gamma^{(0)},\ldots,\gamma^{(de-1)})$) is called the $p$-weight of
$B$ (respectively the $p$-core of $B$). Note that $B$ has $p$-defect 0 if and only if $w=(0, \ldots, 0)$. We denote by $\mathcal E_{\uga,\uw}$
the set of $de$-multipartitions $\umu=(\mu^{(0)},\ldots,\mu^{(de-1)})$ such that
$(\mu^{(i)})_{(p)}=\gamma^{(i)}$ and $w_p(\mu^{(i)})=w_i$.

\smallskip
We can now describe the $p$-blocks of $N$ using Clifford theory. If $B$ is a $p$-block of $G$ of defect 0, then, since $(p,e)=1$, $B$ only covers $p$-blocks of defect 0 of $N$. Conversely, a $p$-block of $N$ of defect 0 can only been covered by $p$-blocks of $G$ of defect 0. Hence suppose $B$ is a $p$-block of $G$ of positive defect, and take $k$ dividing $e$ minimal such that $B$ is $\e^k$-stable (i.e. $\e^k \otimes B=B$). Then $B$ has $p$-core $\uga=(\ga^{(0)}, \ga^{(1)}, \ldots, \ga^{(kd-1)}, \ga^{(0)},\ldots, \ga^{(kd-1)}, \ldots , \ga^{(0)},\ldots, \ga^{(kd-1)} )$ and $p$-weight $\uw=(w_0, w_1, \ldots, w_{kd-1}, w_0, \ldots, w_{kd-1}, \ldots, w_0, \ldots, w_{kd-1})$, where $w_0 + \cdots + w_{kd-1} \neq 0$. Without loss of generality, we can furthermore suppose that $w_0 \neq 0$. Now consider any $\ula \in {\cal{MP}}_{r,de}$ given by
$$\ula=(\la^{(0)}, \la^{(1)}, \ldots , \la^{(kd-1)}, \mu^{(0)}, \la^{(1)}, \ldots , \la^{(kd-1)}, 
% \mu^{(0)}, \la^{(1)}, \ldots , \la^{(kd-1)},
 \ldots,  \mu^{(0)}, \la^{(1)}, \ldots , \la^{(kd-1)}),$$where
\begin{itemize}
\item
for $1 \leq i \leq kd-1$, $(\la^{(i)})_{(p)}=\ga^{(i)}$ and $w_p(\la^{(i)})=w_i$,

\item
$(\la^{(0)})_{(p)}=(\mu^{(0)})_{(p)}=\ga^{(0)}$, $\la^{(0)}$ and $\mu^{(0)}$ have $p$-quotients $Q_p(\la^{(0)})=((w_0), \emptyset, \ldots, \emptyset)$ and $Q_p(\mu^{(0)})=( \emptyset, \ldots, \emptyset, (w_0))$ (so that $\la^{(0)} \neq \mu^{(0)}$),

\item
and, for $1 \leq j \leq k-1$, $$\la^{(jd)}= \left\{ \begin{array}{cl} \mu^{(0)} & \mbox{if} \; w_{jd}=w_0 \; \mbox{and} \; \ga^{(jd)}=\ga^{(0)} \\ \mbox{any} \; \mu \; \mbox{with}  \; \mu_{(p)}=\ga^{(jd)} \; \mbox{and} \; w_p(\mu)= w_{jd} &  \mbox{if} \; w_{jd} \neq w_0 \; \mbox{or} \; \ga^{(jd)} \neq \ga^{(0)} \end{array} \right. $$
(so that $\la^{(jd)} \neq \la^{(0)}$).
\end{itemize}

Then $\ula_{(p)}=\uga$ and $w_p(\ula)=\uw$, so that $\widetilde{\chi}_{\ula} \in B$. And $\la^{(jd)} \neq \la^{(0)}$ for all $0 < j <e$, so that $\e^j(\ula)\neq \ula$, and $\widetilde{\chi}_{\ula}$ is not $\e^j$-stable for any $0 < j < e$.

This shows that any $p$-block $B$ of $G$ of positive defect contains an irreducible character which is not $\e^j$-stable for any $0 < j < e$. By Clifford theory, such a character must restrict irreducibly to $N$, and its restriction to $N$ is therefore $G$-stable. By \cite[Corollary (9.3)]{Navarro}, this implies that $B$ covers a unique $p$-block $b$ of $N$.

 \subsection{Bijections and isometries between blocks}
 \label{sec:bijections}
 
 We now fix the positive integers $d$ and $e$, a prime $p$ not dividing $de$, and consider two positive integers $r$ and $r'$. We let $G=G(de,1,r)$, $N=G(de,e,r)$, $G'=G(de,1,r')$ and $N'=G(de,e,r')$. Suppose $b$ is a $p$-block of $N$, covered by the $p$-block $B$ of $G$ of $p$-core $\uga=(\ga^{(0)}, \ldots , \ga^{(de-1)})$ and $p$-weight $\uw=(w_0, \ldots, w_{de-1})$, and $b'$ is a $p$-block of $N'$, covered by the $p$-block $B'$ of $G'$ of $p$-core $\uga'=(\ga'^{(0)}, \ldots , \ga'^{(de-1)})$ and $p$-weight $\uw'=\uw$. Suppose furthermore that $w_0 + \cdots + w_{de-1} \neq 0$. Then there is a bijection $\psi$ between the subsets $\mathcal E_{\uga,\uw}$ and $\mathcal E_{\uga',\uw}$ of ${\cal{MP}}_{r,de}$ and ${\cal{MP}}_{r',de}$ (which parametrize the irreducible characters in $B$ and $B'$ respectively) described as follows. For any $\ula=(\la^{(0)},\ldots,\la^{(de-1)}) \in \mathcal E_{\uga,\uw}$, we have $\psi(\ula)=(\Psi(\la^{(0)}), \ldots ,\Psi(\la^{(de-1)}))$, where, for each $0 \leq i \leq de-1$, $\Psi(\la^{(i)})$ is the partition defined by $\Psi(\la^{(i)})_{(p)}=\ga'^{(i)}$ and $Q_p(\Psi(\la^{(i)}))=Q_p(\la^{(i)})$.
 
With the notation of Section \ref{sec:representations}, we see that, for any $\ula \in \mathcal E_{\uga,\uw}$, we have $|C_{\psi(\ula)}|=|C_{\ula}|$ and $b_{\psi(\ula)}=b_{\ula}$. Furthermore, for any $\ula, \umu \in \mathcal E_{\uga,\uw}$, we have, with the notation of Section \ref{sec:carN}, $[\ula]=[\umu]$ if and only if $[\psi(\ula)]=[\psi(\umu)]$. In particular, $\psi$ also induces a bijection between $\Irr(b)$ and $\Irr(b')$.

Before our next definition, we need a few more pieces of notation. If $s$ and $t$ are positive integers, and if $n \in \N$, then, for any $\alpha \in {\cal{MP}}_{s,n}$, we set $t\alpha=(\alpha, \ldots, \alpha) \in {\cal{MP}}_{ts,tn}$. If $\beta=t\alpha$, then we write $\alpha=\beta/t$. 

\noindent
Finally, for any $k>0$ and any $k$-multipartition $\ula=(\la^{(0)}, \la^{(1)}, \ldots, \la^{(k-1)})$, we set $\da_p(\ula)=\da_p(\la^{(0)})\da_p(\la^{(1)}) \cdots \da_p(\la^{(k-1)})$, where, for each $0 \leq i <k$, $\da_p(\la^{(i)})$ is the $p$-sign of $\la^{(i)}$ (see \cite[\textsection 2]{morrisolsson}).

\begin{definition}
\label{def:perfisom}
With the notation above, we define an isometry $I \colon \C \Irr(b) \longrightarrow \C \Irr(b')$ by letting, for any $\ula \in \mathcal E_{\uga,\uw}$ and any $0 \leq i < | C_{\ula}|$,
$$I(\chi_{\ula,i})= \left\{ \begin{array}{ll}  \da_p(\ula) \da_p(\psi(\ula)) \chi_{\psi(\ula),i} & \mbox{if} \; | C_{\ula} | \; \mbox{is odd}, \\ \da_p(\ula) \da_p(\psi(\ula)) \chi_{\psi(\ula),i,\da_{\ula} M_{\psi(\ula)}} & \mbox{if} \; | C_{\ula} | \; \mbox{is even}, \end{array} \right. $$
where $\da_{\ula}= \da_p(\ula/ |C_{\ula}|)\da_p(\psi(\ula)/|C_{\psi(\ula)}|)$.

\end{definition}

%  {\color{red}Remarque: plus loin, dans Lemma \ref{lem:imagedelta}, je
%  montre/dis que, si $| C_{\ula} |$ est pair, alors $\da_p(\ula)
%  \da_p(\psi(\ula))=1$. Du coup, est-ce que ça vaut vraiment la peine de
%  se le tra\^iner? L'intér\^et majeur, c'est que ça parait plus uniforme
%  avec, mais bon...}

\begin{remark}
\label{rem:welldefined}
Note that $I$ is well-defined. Indeed, if $[\ula]=[\umu]$, then $[\psi(\ula)]=[\psi(\umu)]$, so that $\da_p(\ula)=\da_p(\umu)$ and  $\da_p(\psi(\ula)) = \da_p(\psi(\umu))$. Also, by Lemma \ref{lem:carorbites}, $\chi_{\ula,i}=\chi_{\umu,i}$ for all $0 \leq i < |C_{\ula}|$. Furthermore, by the proof of Lemma \ref{lem:carorbites}, $M_{\psi(\ula)}=M_{\psi(\umu)}$ and thus $\da_{\ula} M_{\psi(\ula)}=\da_{\umu} M_{\psi(\umu)}$ (since $\da_{\ula}=\da_{\umu}$). Finally, by Remark \ref{rem:signe} and Lemma \ref{lem:carorbites}, $\chi_{\psi(\ula),i,\da_{\ula} M_{\psi(\ula)}} =\chi_{\psi(\umu),i,\da_{\umu} M_{\psi(\umu)}} $ for all $0 \leq i < |C_{\ula}|$.

\end{remark}

\subsection{Perfect isometries}
\label{sec:isomparf}

We keep the notation as in the previous section. Our aim is now to show that the isometry $I$ described in Definition \ref{def:perfisom} is actually a perfect isometry between $b$ and $b'$, thereby generalizing to complex reflection groups the results known about the symmetric groups (see \cite[Theorem 11]{Enguehard}), wreath products (see \cite[Theorem 5.4]{BrGr3}) and Weyl groups of type $B$ and $D$ (see \cite[Corollary 5.6 and Theorem 5.8]{BrGr3}). We start by recalling the definition of perfect isometry.

\begin{definition}(See \cite{Broue} and \cite[\textsection 2.5]{BrGr3})
\label{def:perfectisometry}
Let $H$ and $H'$ be finite groups, $p$ be a prime, and $(K, {\cal R}, k)$ a splitting $p$-modular system for $H$ and $H'$. Let ${\cal B}$ and ${\cal B}'$ be unions of $p$-blocks of $H$ and $H'$ respectively, and $J \colon \C \Irr({\cal B}) \longrightarrow \C \Irr({\cal B}')$ an isometry such that $J( \Z \Irr({\cal B}))=\Z \Irr({\cal B}')$. Let $(e_1, \ldots, e_n)$ be any $\C$-basis for $\C \Irr({\cal B})$ and $(e_1^{\vee}, \ldots,e_n^{\vee})$ its dual with respect to the usual hermitian product $\langle \; , \; \rangle_H$ on $\C \Irr(H)$, and let $\widehat{J}=\dis \sum_{i=1}^n \overline{e_i^{\vee}} \otimes J(\e_i)$. Then $J$ is a perfect isometry between ${\cal B}$ and ${\cal B}'$ if the following hold:

\begin{enumerate}

\item[(1)]
For every $(x,x') \in H \times H'$, $\widehat{J}(x,x') \in | C_H(x)|_p {\cal R} \cap | C_{H'}(x')|_p {\cal R} $.

\item[(2)]
If $\widehat{J}(x,x') \neq 0$, then $x$ and $x'$ are both $p$-regular or both $p$-singular.

\end{enumerate}
\end{definition}

\begin{remark}
Note that, in Definition \ref{def:perfectisometry}, $\widehat{J}$ does not in fact depend on the choice of basis for $\C \Irr({\cal B})$ (see \cite[\textsection 2.3]{BrGr3}).
\end{remark}

If we let $[\mathcal E_{\uga,\uw}]$ be a set of representatives for the $\cyc{\e}$-orbits of $\mathcal E_{\uga,\uw}$, then $\{\chi_{\ula,i}, \; \ula \in [\mathcal E_{\uga,\uw}] \; \mbox{and} \; 0 \leq i < | C_{\ula} | \}$ is a (self-dual) $\C$-basis for $\C \Irr(b)$. By (\ref{eq:chardiff}), $\{\Delta_{\ula,i}, \; \ula \in [\mathcal E_{\uga,\uw}] \; \mbox{and} \; 0 \leq i < | C_{\ula} | \}$ is also a $\C$-basis for $\C \Irr(b)$, and this is the basis we will use to prove that $I$ is a perfect isometry between $b$ and $b'$.

From (\ref{eq:chardiff}), we see that, for any $\ula , \umu \in [\mathcal E_{\uga,\uw}]$, $0 \leq i < | C_{\ula} |$ and $0 \leq j < | C_{\umu} |$, we have
$$\begin{array}{rcl} \langle \Delta_{\ula,i},\Delta_{\umu,j} \rangle_N & = & \langle \dis \sum_{k=0}^{|C_{\ula}|-1} \zeta^{d b_{\ula}ik} \chi_{\ula,k} , \sum_{\ell=0}^{|C_{\umu}|-1} \zeta^{d b_{\umu}j\ell} \chi_{\umu,\ell} \rangle_N \\ & = & \dis \sum_{k=0}^{|C_{\ula}|-1} \sum_{\ell=0}^{|C_{\umu}|-1} \zeta^{d b_{\ula}ik}   \overline{\zeta^{d b_{\umu}j\ell}} \langle \chi_{\ula,k} , \chi_{\umu,\ell} \rangle_N
\\
 & = &  \dis \sum_{k=0}^{|C_{\ula}|-1} \sum_{\ell=0}^{|C_{\umu}|-1} \zeta^{d b_{\ula}ik-d b_{\umu}j\ell} \da_{\ula,\umu} \da_{k,\ell} 
 \\
  & = & \da_{\ula,\umu} \dis \sum_{k=0}^{|C_{\ula}|-1}  \zeta^{d b_{\ula}(i-j)k} 
  \\
  & = &  \da_{\ula,\umu} \da_{i,j} |C_{\ula}|.
\end{array}$$
This shows that, for any $\ula \in [\mathcal E_{\uga,\uw}]$ and $0 \leq i < | C_{\ula} |$, we have 
\begin{equation}
\label{eq:dualdelta}
\Delta_{\ula,i}^{\vee}=\dis \frac{1}{| C_{\ula} |} \Delta_{\ula,i}.
\end{equation}

Furthermore, from (\ref{eq:chardiff}) and Definition \ref{def:perfisom}, we see that, for any $\ula \in \mathcal E_{\uga,\uw}$ and $0 \leq i < | C_{\ula} |$, we have
$$\begin{array}{rcl}
I(\Delta_{\ula,i}) & = & \dis \sum_{j=0}^{| C_{\ula} | -1} \zeta^{d b_{\ula}ij} I (\chi_{\ula,j}) 
\\
 & = & \left\{ \begin{array}{ll}  \da_p(\ula) \da_p(\psi(\ula)) \dis \sum_{j=0}^{| C_{\ula} | -1} \zeta^{d b_{\ula}ij} \chi_{\psi(\ula),j} & \mbox{if} \; | C_{\ula} | \; \mbox{is odd}, \\ \da_p(\ula) \da_p(\psi(\ula)) \dis \sum_{j=0}^{| C_{\ula} | -1} \zeta^{d b_{\ula}ij}  \chi_{\psi(\ula),j,\da_{\ula} M_{\psi(\ula)}} & \mbox{if} \; | C_{\ula} | \; \mbox{is even}. \end{array} \right. 
\end{array}$$
Now, since $|C_{\ula}|=|C_{\psi(\ula)}|$ and $b_{\ula}=b_{\psi(\ula)}$, we have that, if $|C_{\ula}|$ is odd, then
$$\dis \sum_{j=0}^{| C_{\ula} | -1} \zeta^{d b_{\ula}ij} \chi_{\psi(\ula),j} = \dis \sum_{j=0}^{| C_{\psi(\ula)} | -1} \zeta^{d b_{\psi(\ula)}ij} \chi_{\psi(\ula),j}= \Delta_{\psi(\ula),i}$$
(by (\ref{eq:chardiff})), and, if $|C_{\ula}|$ is even, then
$$\dis \sum_{j=0}^{| C_{\ula} | -1} \zeta^{d b_{\ula}ij}  \chi_{\psi(\ula),j,\da_{\ula} M_{\psi(\ula)}}= \dis \sum_{j=0}^{| C_{\psi(\ula)} | -1} \zeta^{d b_{\psi(\ula)}ij}  \chi_{\psi(\ula),j,\da_{\ula} M_{\psi(\ula)}}=\Delta_{\psi(\ula),i,\da_{\ula} M_{\psi(\ula)}}$$
(by Remark \ref{rem:signe}). And, also by Remark \ref{rem:signe}, we have $\Delta_{\psi(\ula),i,\da_{\ula} M_{\psi(\ula)}}= \da_{\ula}^i \Delta_{\psi(\ula),i}$.

This shows that, for any $\ula \in \mathcal E_{\uga,\uw}$ and $0 \leq i < | C_{\ula} |$, we have
\begin{equation}
\label{eq:imagedelta}
I(\Delta_{\ula,i})  =  \left\{ \begin{array}{ll}  \da_p(\ula) \da_p(\psi(\ula)) \Delta_{\psi(\ula),i} & \mbox{if} \; | C_{\ula} | \; \mbox{is odd}, \\ \da_p(\ula) \da_p(\psi(\ula))  \da_{\ula}^i \Delta_{\psi(\ula),i} & \mbox{if} \; | C_{\ula} | \; \mbox{is even}. \end{array} \right. 
\end{equation}

When we compute $\widehat{I}$, we will regroup characters $\Delta_{\ula,i}$ in ``slices'' according to the order modulo $e$ of the integer $b_{\ula}i$. First note that, as an additive group, we have $$\Z / e \Z = \dis \coprod_{q | e} \{ \bar{k} \in \Z / e \Z \; | \; \mbox{ord}(\bar{k})=q \}=\coprod_{q | e} \left\{ \overline{\left( \frac{e}{q}s \right)} \; | \; 0 \leq s<q \; \mbox{and} \; (s,q)=1 \right\}.$$
Since, whenever $0 \leq s <q$, we have $0 \leq \dis \frac{e}{q}s<e$, we actually obtain$$\{0, \, \ldots , \, e-1 \} = \coprod_{q | e}  \left\{ \dis \frac{e}{q}s \; | \; 0 \leq s<q \; \mbox{and} \; (s,q)=1 \right\}.$$
Our ``slices'' are described by the following.

\begin{proposition}
\label{prop:slices}
For any $0 \leq k \leq e-1$, we let $${\cal P}_{\uga,\uw,k} = \{ (\ula,i) \; | \; \ula \in \mathcal E_{\uga,\uw} \; , \; 0 \leq i < | C_{\ula} | \; \mbox{and} \; b_{\ula}i=k\}.$$
Let $q$ be the order of $k$ modulo $e$. Then the maps $\alpha \colon {\cal P}_{\uga,\uw,k} \longrightarrow \mathcal E_{\uga/q,\uw/q}$ and $\beta \colon \mathcal E_{\uga/q,\uw/q} \longrightarrow  {\cal P}_{\uga,\uw,k}$ given by $\alpha ( (\ula,i))=\ula/q$ and $\beta(\umu)=(q\umu, k/b_{\umu})$ are mutually inverse bijections.
\end{proposition}
\begin{remark}
\label{rem:slices}
Recall that $\mathcal E_{\uga/q,\uw/q}$ is exactly the set of multipartitions labelling the irreducible characters which belong to the $p$-block of $G(de/q,1,r/q)$ with $p$-core $\uga/q$ and $p$-weight $\uw/q$. 
\end{remark}
\begin{proof}
We start by showing that $\alpha$ and $\beta$ are indeed defined.

If $(\ula,i) \in {\cal P}_{\uga,\uw,k} $, then $b_{\ula}i=k$. Since $k=\dis \frac{e}{q}s$ for some $0 \leq s <q$ with $(s,q)=1$, we have $b_{\ula}i=\dis \frac{e}{q}s$. Hence $qi=s \dis \frac{e}{b_{\ula}}=s | C_{\ula} |$. Since $(s,q)=1$, this shows that $q$ divides $|C_{\ula}|$. Hence $\ula/q$ is indeed defined, and so are $\uga/q$ and $\uw/q$, and $\ula/q$ certainly has $p$-core $\uga/q$ and $p$-weight $\uw/q$. Thus the map $\alpha \colon {\cal P}_{\uga,\uw,k} \longrightarrow \mathcal E_{\uga/q,\uw/q}$ is defined.

If, on the other hand, $\umu \in \mathcal E_{\uga/q,\uw/q}$, then $q \umu$ certainly is defined, and $q \umu \in \mathcal E_{\uga,\uw}$. Furthermore, by definition, $b_{\umu}=\dis \frac{e/q}{| C_{\umu} |}$, so that $b_{\umu} $ divides $e/q$, and also $(e/q)s=k$, whence $k/b_{\umu}$ is an integer. Moreover, we have $|C_{q \umu}|=q . |C_{\umu}|$, and, since $0 \leq k <e$, we have$$0 \leq \dis \frac{k}{b_{\umu}} < \frac{e}{b_{\umu}}=q \frac{e/q}{b_{\umu}}=q. |C_{\umu}|=|C_{q \umu}|.$$
Finally, since $b_{\umu}= \dis \frac{e/q}{|C_{\umu}|}= \frac{e}{q.|C_{\umu}|}=\frac{e}{|C_{q \umu}|}=b_{q \umu}$, we have $b_{q \umu} \dis \frac{k}{b_{\umu}}=k$, whence $(q\umu, k/b_{\umu})$ is defined, and $(q\umu, k/b_{\umu}) \in {\cal P}_{\uga,\uw,k} $.

It only remains to show that $\alpha$ and $\beta$ are mutual inverses. For any $\umu \in \mathcal E_{\uga/q,\uw/q}$, we have $(\alpha \circ \beta)(\umu)=\alpha ((q\umu, k/b_{\umu}))=q \umu/q=\umu$. And, for any $(\ula,i) \in {\cal P}_{\uga,\uw,k}$, we have
$$(\beta \circ \alpha)((\ula,i))=\beta(\ula/q)=(q \ula/q,\dis \frac{k}{b_{\ula/q}})=(\ula,\dis \frac{k}{b_{\ula/q}})=(\ula,\dis \frac{k}{b_{\ula}})$$
(since, as we've seen above, $b_{\ula/q}=b_{\ula}$). Finally, since $(\ula,i) \in {\cal P}_{\uga,\uw,k}$, we have $\dis \frac{k}{b_{\ula}}=i$, whence $(\beta \circ \alpha)((\ula,i))=(\ula,i)$. This concludes the proof.

\end{proof}

\begin{remark}
\label{rem:ordredei}
Note that, with the notation of Proposition \ref{prop:slices}, if $(\ula,i) \in {\cal P}_{\uga,\uw,k} $ and $k=b_{\ula}i$ has order $q$ modulo $e$ (written $\mbox{ord}_e(k)=q$), then $i$ has order $q$ modulo $e/b_{\ula}=|C_{\ula}|$ (written $\mbox{ord}_{|C_{\ula}|}(i)=q$). Hence there exists $s$ such that $(s,q)=1$ and $i=s | C_{\ula} |/q$. Suppose furthermore that $| C_{\ula}|$ is even. If $q$ is even, then $(s,q)=1$ implies that $s$ is odd, so that $| C_{\ula} |/q$ and $i=s | C_{\ula} |/q$ have the same parity. If, on the other hand, $q$ is odd, then, since $| C_{\ula}|$ is even, $| C_{\ula} |/q$ is even, and so is $i=s | C_{\ula} |/q$. This shows that, whenever $| C_{\ula}|$ is even, $| C_{\ula} |/q$ and $i$ have the same parity. Now we have $\ula/q=(| C_{\ula} |/q) \cdot \ula/ | C_{\ula} |$. Taking $p$-signs, we have $\da_p(\ula/q)=\da_p((| C_{\ula} |/q) \cdot \ula/ | C_{\ula} |)=\da_p(\ula/ | C_{\ula} |)^{| C_{\ula} |/q}$. And, since $| C_{\ula} |/q$ and $i$ have the same parity, we obtain
\begin{equation}
\label{eq:signe}
\da_p(\ula/ | C_{\ula} |)^i=\da_p(\ula/q) \; \mbox{whenever} \; (\ula,i) \in {\cal P}_{\uga,\uw,k} \; \mbox{and} \; \mbox{ord}_e(k)=q.
\end{equation}

\end{remark}
From this, we easily deduce the following.

\begin{lemma}
\label{lem:imagedelta}
If $I$ is the map described in Definition \ref{def:perfisom} and $(\ula,i) \in {\cal P}_{\uga,\uw,k} $, where $k=b_{\ula}i$ has order $q$ modulo $e$, then $I(\Delta_{\ula,i})=\da_p(\ula/q)\da_p(\psi(\ula)/q) \Delta_{\psi(\ula),i}$.
\end{lemma}

\begin{proof}
By (\ref{eq:imagedelta}), we know that
$$I(\Delta_{\ula,i})  =  \left\{ \begin{array}{ll}  \da_p(\ula) \da_p(\psi(\ula)) \Delta_{\psi(\ula),i} & \mbox{if} \; | C_{\ula} | \; \mbox{is odd}, \\ \da_p(\ula) \da_p(\psi(\ula))  \da_{\ula}^i \Delta_{\psi(\ula),i} & \mbox{if} \; | C_{\ula} | \; \mbox{is even}. \end{array} \right. $$
If $| C_{\ula} |=| C_{\psi(\ula)} |$ is odd, then $q$ must be odd (since $q$ divides $| C_{\ula} |$), so that $\da_p(\ula) = \da_p(q \cdot \ula/q) = \da_p(\ula/q)^q=\da_p(\ula/q)$ and $\da_p(\psi(\ula))=\da_p(q \cdot \psi(\ula)/q)=\da_p(\psi(\ula)/q)^q=\da_p(\psi(\ula)/q)$. Hence, in this case, $\da_p(\ula) \da_p(\psi(\ula))=\da_p(\ula/q) \da_p(\psi(\ula)/q)$.

If, on the other hand, $| C_{\ula} |=| C_{\psi(\ula)} |$ is even, then $\da_p(\ula)=\da_p(|C_{\ula}| \cdot \ula / |C_{\ula}|)=\da_p(\ula / |C_{\ula}|)^{|C_{\ula}|}=1$ and $\da_p(\psi(\ula))=\da_p(|C_{\ula}| \cdot \psi(\ula) / |C_{\ula}|)=\da_p(\psi(\ula) / |C_{\ula}|)^{|C_{\ula}|}=1$. And, by (\ref{eq:signe}) (and since $(\psi(\ula),i) \in {\cal P}_{\uga',\uw,k} $), $$ \da_{\ula}^i= \da_p(\ula/ |C_{\ula}|)^i \da_p(\psi(\ula)/|C_{\psi(\ula)}|)^i=\da_p(\ula/q)\da_p(\psi(\ula)/q).$$
Hence, in this case, $\da_p(\ula) \da_p(\psi(\ula))  \da_{\ula}^i=\da_p(\ula/q)\da_p(\psi(\ula)/q)$.
\end{proof}

We can now state and prove our main result.

\begin{theorem}
\label{thm:perfisom}
Take any positive integers $d$, $e$, $r$ and $r'$, and a prime $p$ not dividing $de$. Let $G=G(de,1,r)$, $N=G(de,e,r)$, $G'=G(de,1,r')$ and $N'=G(de,e,r')$. Suppose $b$ is a $p$-block of $N$, covered by the $p$-block $B$ of $G$ of $p$-core $\uga$ and $p$-weight $\uw$, and $b'$ is a $p$-block of $N'$, covered by the $p$-block $B'$ of $G'$ of $p$-core $\uga'$ and $p$-weight $\uw'=\uw$. Then there is a perfect isometry between $b$ and $b'$.
\end{theorem}

\begin{proof}
First note that, if $\uw=(0, \ldots 0)$, then both $b$ and $b'$ are $p$-blocks of defect 0, so that $b=\{ \chi \}$ and $b'=\{ \chi'\}$ for some irreducible characters $\chi$ and $\chi'$ (of $N$ and $N'$ respectively) which vanish on $p$-singular elements. If we define $I \colon \C \Irr(b) \longrightarrow \C \Irr(b')$ by $I(\chi)=\chi'$, then, with the notation of Definition \ref{def:perfectisometry}, we have $\widehat{I}=\overline{\chi} \otimes \chi'$. Since $\chi$ and $\chi'$ vanish on $p$-singular elements, we have $\widehat{I}(x,x')=\overline{\chi(x)}\chi'(x') \neq 0$ only if $x$ and $x'$ are both $p$-regular, so that property (2) of Definition \ref{def:perfectisometry} holds. Furthermore, since $b=\{ \chi \}$ and $b'=\{ \chi'\}$, $\chi$ and $\chi'$ are actually projective indecomposable characters (of $N$ and $N'$ respectively). Hence, by 
\cite[Lemma (2.21)]{Navarro}, for all $(x,x') \in N \times N'$, $\dis \frac{\chi(x)}{| C_N(x)|_p} \in {\cal R}$ and $\dis \frac{\chi'(x')}{| C_{N'}(x')|_p} \in {\cal R}$. Property (1) of Definition \ref{def:perfectisometry} immediately follows. This shows that, if $\uw=(0, \ldots 0)$, then $b$ and $b'$ are perfectly isometric.

We therefore now suppose that $\uw \neq (0, \ldots 0)$. Let $I \colon \C \Irr(b) \longrightarrow \C \Irr(b')$ be the map described in Definition \ref{def:perfisom}. We will decompose $\widehat{I}$ using the $\C$-basis $\{\Delta_{\ula,i}, \; \ula \in [\mathcal E_{\uga,\uw}] \; \mbox{and} \; 0 \leq i < | C_{\ula} | \}$ for $\C \Irr(b)$. We have, by Definition \ref{def:perfectisometry}, $ e \widehat{I}  =  e \dis \sum_{\ula \in [\mathcal E_{\uga,\uw}]} \sum_{i=0}^{| C_{\ula} | -1} \overline{\Delta_{\ula,i}^{\vee}} \otimes I(\Delta_{\ula,i}) $, so that, by (\ref{eq:dualdelta}),
$$ e \widehat{I}  = e \dis \sum_{\ula \in [\mathcal E_{\uga,\uw}]} \sum_{i=0}^{| C_{\ula} | -1} \frac{1}{| C_{\ula} |} \overline{\Delta_{\ula,i}} \otimes I(\Delta_{\ula,i}).$$
Since $| [ \ula ] |=b_{\ula}$ and $b_{\ula}| C_{\ula} |=e$, Corollary \ref{cor:cardifforbites} gives
$$ e \widehat{I}  =  \dis \sum_{\ula \in \mathcal E_{\uga,\uw}} \frac{1}{b_{\ula}} \sum_{i=0}^{| C_{\ula} | -1} \frac{1}{| C_{\ula} |} \overline{\Delta_{\ula,i}} \otimes I(\Delta_{\ula,i})= \dis \sum_{\ula \in \mathcal E_{\uga,\uw}} \sum_{i=0}^{| C_{\ula} | -1}  \overline{\Delta_{\ula,i}} \otimes I(\Delta_{\ula,i}) .$$
Using our ``slices'', we obtain

\begin{eqnarray*} e \widehat{I} & = & \dis \sum_{k=0}^{e-1} \; \; \; \sum_{(\ula,i) \in {\cal P}_{\uga,\uw,k} } \overline{\Delta_{\ula,i}} \otimes I(\Delta_{\ula,i}) 
\\
& = & \dis \sum_{q | e} \; \; \sum_{0 \leq k< e \atop \mbox{\tiny{ord}}_e(k)=q}  \; \; \sum_{(\ula,i) \in {\cal P}_{\uga,\uw,k} } \overline{\Delta_{\ula,i}} \otimes I(\Delta_{\ula,i}) 
\\
& = & \dis \sum_{q | e}  \sum_{0 \leq k< e \atop \mbox{\tiny{ord}}_e(k)=q}  \sum_{(\ula,i) \in {\cal P}_{\uga,\uw,k} } \overline{\Delta_{\ula,i}} \otimes \da_p(\ula/q)\da_p(\psi(\ula)/q) \Delta_{\psi(\ula),i}
\end{eqnarray*}
(by Lemma \ref{lem:imagedelta}).

\medskip

Now take any $(x,x') \in N \times N'$. Write $x= {}^gg_{\ueta}$ and $x'={}^{g'}g_{\ueta'}$, where $\ueta \in {\cal{MP}}_{r,de}$, $\ueta' \in {\cal{MP}}_{r',de}$, $g_{\ueta} \in N$ and $g_{\ueta'}\in N'$ are as in Convention \ref{conv:defrep}, $g \in G$ and $g' \in G'$.

\noindent
Take any $0 \leq k \leq e-1$, and $(\ula,i) \in {\cal P}_{\uga,\uw,k}$. 
%Recall that $M_{\ula} \rho_{\ula} M_{\ula}^{-1}=\e^{b_{\ula}} \otimes \rho_{\ula}$, so that $M_{\ula}^i \rho_{\ula} M_{\ula}^{-i}=\e^{b_{\ula}i} \otimes \rho_{\ula}=\e^{k} \otimes \rho_{\ula}$ (since $(\ula,i) \in {\cal P}_{\uga,\uw,k}$). We thus have
We have $\varepsilon^k=\varepsilon^{b_{\underline\lambda}i}$
(since $(\ula,i) \in {\cal P}_{\uga,\uw,k}$). Hence,
Proposition~\ref{prop:conjvaleurs} gives
%  \begin{eqnarray*}
%  \Delta_{\ula,i}(x)=\Delta_{\ula,i}({}^gg_{\ueta}) & = & \Tr (M_{\ula}^i \rho_{\ula} ({}^gg_{\ueta}))
%  \\
%   & = &  \Tr (M_{\ula}^i \rho_{\ula} (g) \rho_{\ula} (g_{\ueta}) \rho_{\ula} (g)^{-1})
%   \\
%   & = &  \Tr (\e^{k} \otimes \rho_{\ula}(g) M_{\ula}^i \rho_{\ula} (g_{\ueta}) \rho_{\ula} (g)^{-1})
%  \\
%  & = & \Tr (\e^{k}(g) \rho_{\ula}(g) M_{\ula}^i \rho_{\ula} (g_{\ueta}) \rho_{\ula} (g)^{-1})
%  \\
%  & = & \e^{k}(g) \Tr ( \rho_{\ula}(g) M_{\ula}^i \rho_{\ula} (g_{\ueta}) \rho_{\ula} (g)^{-1})
%  \\
%  & = & \e^{k}(g) \Tr (M_{\ula}^i \rho_{\ula} (g_{\ueta})),
%  \end{eqnarray*}
%  whence
\begin{equation}
\label{eq:rootofunity} 
\Delta_{\ula,i}(x)=\e^{k}(g)\Delta_{\ula,i}(g_{\ueta}).
\end{equation}
Similarly, since $(\psi(\ula),i) \in {\cal P}_{\uga',\uw,k}$, we have
\begin{equation}
\label{eq:rootofunity'} 
\Delta_{\psi(\ula),i}(x')=\e^{k}(g')\Delta_{\psi(\ula),i}(g_{\ueta'}).
\end{equation}

Now, if $\mbox{ord}_e(k)=q$, then, for all $(\ula,i) \in {\cal P}_{\uga,\uw,k}$, we have, by Theorem \ref{thm:memeordre},
\begin{equation}
\label{eq:delta}
\Delta_{\ula,i}(g_{\ueta}) = \Delta_{\ula,|C_{\ula}|/q}(g_{\ueta}).
\end{equation}
Furthermore, still supposing $\mbox{ord}_e(k)=q$, Proposition \ref{prop:slices} (applied twice) shows that
\begin{equation}
\label{eq:Q}
Q \colon \left\{ \begin{array}{ccc} {\cal P}_{\uga,\uw,k} & \longrightarrow & {\cal P}_{\uga,\uw,e/q}  \\ (\ula,i) & \longmapsto & (\ula, |C_{\ula}|/q) \end{array} \right. \; \;  \mbox{is a bijection.}
\end{equation}

Similarly, for all $(\ula,i) \in {\cal P}_{\uga,\uw,k}$, we have, by Theorem \ref{thm:memeordre},
\begin{equation}
\label{eq:delta'}
\Delta_{\psi(\ula),i}(g_{\ueta'}) = \Delta_{\psi(\ula),|C_{\ula}|/q}(g_{\ueta'})
\end{equation}
and, by Proposition \ref{prop:slices},
\begin{equation}
\label{eq:Q'}
Q' \colon \left\{ \begin{array}{ccc} {\cal P}_{\uga',\uw,k} & \longrightarrow & {\cal P}_{\uga',\uw,e/q}  \\ (\psi(\ula),i) & \longmapsto & (\psi(\ula), |C_{\ula}|/q) \end{array} \right. \; \;  \mbox{is a bijection.}
\end{equation}

Note also that, by Proposition \ref{prop:slices},
\begin{equation}
\label{eq:alpha}
S \colon \left\{ \begin{array}{ccc} {\cal P}_{\uga,\uw,e/q} & \longrightarrow & {\cal E}_{\uga/q,\uw/q}  \\ (\ula, |C_{\ula}|/q) & \longmapsto & \ula/q \end{array} \right. \quad \mbox{and}
\end{equation}
$$
S' \colon \left\{ \begin{array}{ccc} {\cal P}_{\uga',\uw,e/q} & \longrightarrow & {\cal E}_{\uga'/q,\uw/q}  \\ (\psi(\ula), |C_{\ula}|/q) & \longmapsto & \psi(\ula)/q = \psi(\ula/q)  \end{array} \right. \; \;  \mbox{are bijections.}$$

Write $g_{\ueta}=(z;\sigma)$ with $\sigma=\sa_1 \cdots \sa_s \in \sym_r$ and $g_{\ueta'}=(z';\sigma')$ with $\sigma'=\sa'_1 \cdots \sa'_{s'} \in \sym_{r'}$ as in Convention \ref{conv:defrep}. By Theorem \ref{th:valcar2}, we see that, if there is $1 \leq u \leq de-1$ such that $q$ does not divide $|\eta_u|$, or if there is $1 \leq j \leq s$ such that $\xi_j \not \in {\cal U}_{de/q}$ (in which case we say that $g_{\ueta}$ is $q$-bad), then, for all $\ula \in {\cal E}_{\uga,\uw}$, we have $\Delta_{\ula,|C_{\ula}|/q}(g_{\ueta})=0$. If, on the other hand, $q$ divide $|\eta_u|$ for all $1 \leq u \leq de-1$ and $\xi_j  \in {\cal U}_{de/q}$ for all $1 \leq j \leq s$ (in which case we say that $g_{\ueta}$ is $q$-good), then, with the notation of Remark \ref{rk:valeurs}, for all $\ula \in {\cal E}_{\uga,\uw}$, we have (by (\ref{eq:valeurs})) 
\begin{equation}
\label{eq:value}
\Delta_{\ula,|C_{\ula}|/q}(g_{\ueta})=q^{s}\widetilde{\chi}_{\ula/q}(g_{\underline\eta}^{(q)}).
\end{equation}
Similarly, if $g_{\ueta'}$ is $q$-bad, then, for all $\ula \in {\cal E}_{\uga,\uw}$, we have $\Delta_{\psi(\ula),|C_{\ula}|/q}(g_{\ueta'})=0$, while, if $g_{\ueta'}$ is $q$-good, then, with the notation of Remark \ref{rk:valeurs}, for all $\ula \in {\cal E}_{\uga,\uw}$, we have (by (\ref{eq:valeurs})) 
\begin{equation}
\label{eq:value'}
\Delta_{\psi(\ula),|C_{\ula}|/q}(g_{\ueta'})=q^{s'}\widetilde{\chi}_{\psi(\ula)/q}(g_{\ueta'}^{(q)}).
\end{equation}

Recall that $e \widehat{I}=\dis \sum_{q | e}  \sum_{0 \leq k< e \atop \mbox{\tiny{ord}}_e(k)=q}  \sum_{(\ula,i) \in {\cal P}_{\uga,\uw,k} } \overline{\Delta_{\ula,i}} \otimes \da_p(\ula/q)\da_p(\psi(\ula)/q) \Delta_{\psi(\ula),i}$.

Writing $\da_p^{\ula,q}$ for $\da_p(\ula/q)\da_p(\psi(\ula)/q)$ whenever $q | e$, we therefore have
$$e \widehat{I}(x,x') = \dis \sum_{q | e}  \sum_{0 \leq k< e \atop \mbox{\tiny{ord}}_e(k)=q}  \sum_{(\ula,i) \in {\cal P}_{\uga,\uw,k} } \da_p^{\ula,q} \overline{\Delta_{\ula,i}(x)}  \Delta_{\psi(\ula),i}(x').$$
By (\ref{eq:rootofunity}) and (\ref{eq:rootofunity'}), this gives
$$e \widehat{I}(x,x') = \dis \sum_{q | e}  \sum_{0 \leq k< e \atop \mbox{\tiny{ord}}_e(k)=q} \overline{\e^k(g)} \e^{k'}(g') \sum_{(\ula,i) \in {\cal P}_{\uga,\uw,k} } \da_p^{\ula,q} \overline{\Delta_{\ula,i}(g_{\ueta})}  \Delta_{\psi(\ula),i}(g_{\ueta}').$$
By (\ref{eq:delta}), (\ref{eq:Q}), (\ref{eq:delta'}) and (\ref{eq:Q'})), we obtain
$$e \widehat{I}(x,x') = \dis \sum_{q | e}  \sum_{0 \leq k< e \atop \mbox{\tiny{ord}}_e(k)=q} \overline{\e^k(g)} \e^{k'}(g') \sum_{(\ula,|C_{\ula}|/q) \in {\cal P}_{\uga,\uw,e/q} } \da_p^{\ula,q} \overline{\Delta_{\ula,|C_{\ula}|/q}(g_{\ueta})}  \Delta_{\psi(\ula),|C_{\ula}|/q}(g_{\ueta}') .$$
Using (\ref{eq:value}) and (\ref{eq:value'})), we get
$$e \widehat{I}(x,x') = \dis \sum_{q | e \atop{g_{\ueta}, g_{\ueta'} \atop \mbox{\tiny{$q$-good}}}}  \sum_{0 \leq k< e \atop \mbox{\tiny{ord}}_e(k)=q} \overline{\e^k(g)} \e^{k'}(g') \sum_{(\ula,|C_{\ula}|/q) \in {\cal P}_{\uga,\uw,e/q}} \da_p^{\ula,q} q^s \overline{\widetilde{\chi}_{\ula/q}(g_{\underline\eta}^{(q)})}  q^{s'}\widetilde{\chi}_{\psi(\ula)/q}(g_{\ueta'}^{(q)})$$
and, by (\ref{eq:alpha}), this finally gives
$$e \widehat{I}(x,x') =\dis \sum_{q | e \atop{g_{\ueta}, g_{\ueta'} \atop \mbox{\tiny{$q$-good}}}} q^s q^{s'}\sum_{0 \leq k< e \atop \mbox{\tiny{ord}}_e(k)=q} \overline{\e^k(g)} \e^{k'}(g') \sum_{\umu \in {\cal E}_{\uga/q,\uw/q} } \da_p(\umu) \da_p(\psi(\umu))  \overline{\widetilde{\chi}_{\umu}(g_{\underline\eta}^{(q)})}  \widetilde{\chi}_{\psi(\umu)}(g_{\ueta'}^{(q)}).$$

We can rewrite this as 
\begin{equation}
\label{eq:Ichapeau}
e \widehat{I}(x,x') =\dis \sum_{q | e \atop{g_{\ueta}, g_{\ueta'} \atop \mbox{\tiny{$q$-good}}}} q^s q^{s'}\sum_{0 \leq k< e \atop \mbox{\tiny{ord}}_e(k)=q} \overline{\e^k(g)} \e^{k'}(g')  \widehat{J}_{q,\uga,\uga',\uw} (g_{\ueta}^{(q)},g_{\ueta'}^{(q)}),
\end{equation}
where, for any $q$ dividing $e$ (and such that $\uga/q$, $\uga'/q$ and $\uw/q$ are defined), $J_{q,\uga,\uga',\uw} $ is the perfect isometry described in \cite[Theorem 5.4]{BrGr3} between the $p$-block $\beta$ of $G(de/q,1,r/q)$ with $p$-core $\uga/q$ and $p$-weight $\uw/q$ and the $p$-block $\beta'$ of $G(de/q,1,r'/q)$ with $p$-core $\uga'/q$ and (same) $p$-weight $\uw/q$.

\medskip
\noindent
We now turn to Properties (1) and (2) of Definition \ref{def:perfectisometry}.

\noindent
Take any $q | e$ such that $g_{\ueta}$ and $g_{\ueta'}$ are $q$-good. Then $\ueta=(\eta_0, \ldots, \eta_{de-1})\in {\cal MP}_{r,de}$, and $\eta_u \neq \emptyset$ only if $q$ divides $u$. Furthermore, if $g_{\ueta}=(z;\sa) \in G(de,1,r)$, then $g_{\ueta}^{(q)}=(z^{(q)};\sa/q) \in G(de/q,1,r/q)$ has cycle type $(\eta_0/q, \eta_q/q, \ldots , \eta_{(de/q-1)q}/q)=(\theta_0, \theta_1, \ldots , \theta_{de/q-1})$ (so that $\theta_i=\eta_{qi}/q$). Note that, since $q | e$ and $(p,e)=1$, $p$ does not divide $q$.

Since $(p,de)=1$, we have that $g_{\ueta}$ is $p$-singular if and only if $\sa \in \sym_r$ is $p$-singular, i.e. if and only if $\sa$ has at least one cycle of length divisible by $p$. Since $p$ does not divide $q$, this is equivalent to $\sa/q$ having at least one cycle of length divisible by $p$. Hence we obtain that $g_{\ueta}$ is $p$-singular if and only $g_{\ueta}^{(q)}$ is $p$-singular. Similarly, $g_{\ueta'}$ is $p$-singular if and only $g_{\ueta'}^{(q)}$ is $p$-singular.

By (\ref{eq:Ichapeau}), if $\widehat{I}(x,x') \neq 0$, then there exists $q$ dividing $e$ such that $g_{\ueta}$ and $g_{\ueta'}$ are $q$-good, and $\widehat{J}_{q,\uga,\uga',\uw} (g_{\ueta}^{(q)},g_{\ueta'}^{(q)})\neq 0$. Since $J_{q,\uga,\uga',\uw} $ is a perfect isometry (by \cite[Theorem 5.4]{BrGr3}), this implies that $g_{\ueta}^{(q)}$ and $g_{\ueta'}^{(q)}$ are both $p$-regular or both $p$-singular, which, by the above, shows that $g_{\ueta}$ and $g_{\ueta'}$, and thus $x$ and $x'$, are both $p$-regular or both $p$-singular. Hence Property (2) of Definition \ref{def:perfectisometry} holds.

\medskip
\noindent
It remains to show that Property (1) holds. First note that, since $|G|/|N|=e$ is coprime to $p$, we have $|C_N(x)|_p=|C_G(x)|_p=|C_G(g_{\ueta})|_p$. Similarly, we have $|C_{N'}(x')|_p=|C_{G'}(g_{\ueta'})|_p$.

Now $|C_{G(de,1,r)}(g_{\ueta})|=\dis \prod_{i,k} \eta_i^{\sharp k}! \, (kde)^{\eta_i^{\sharp k}}$, where $\eta_i^{\sharp k}$ is the number of $k$-cycles in $\eta_i$ (see \cite[Lemma 4.2.10]{James-Kerber}). Since all the cycles in any $\eta_i$ have length divisible by $q$, and since $\eta_u \neq \emptyset$ only if $q$ divides $u$, this can be rewritten as
$$|C_{G(de,1,r)}(g_{\ueta})|=\dis \prod_{i,k} \eta_{qi}^{\sharp qk}! \, (qkde)^{\eta_{qi}^{\sharp qk}}.$$
However, by definition of the cycle type $(\theta_0, \theta_1, \ldots , \theta_{de/q-1})$ of $g_{\ueta}^{(q)}$, we have $\eta_{qi}^{\sharp qk}= (\eta_{qi}/q)^{\sharp k} = \theta_{i}^{\sharp k}$ for all $i$ and $k$. Thus we obtain
$$\begin{array}{rcl} |C_{G(de,1,r)}(g_{\ueta})| & = & \dis \prod_{i,k} \theta_{i}^{\sharp k}! \, (qkde)^{\theta_{i}^{\sharp k}} 
\\ & = & \dis \prod_{i,k} \theta_{i}^{\sharp k}! \, (q^2kde/q)^{\theta_{i}^{\sharp k}} 
\\ & = & \dis q^{2 \sum_{i,k} \theta_{i}^{\sharp k}} \prod_{i,k} \theta_{i}^{\sharp k}! \, (kde/q)^{\theta_{i}^{\sharp k}} 
\\ & = & \dis q^{2s} |C_{G(de/q,1,r/q)}(g_{\ueta}^{(q)})| \qquad \mbox{(where $\sa=\sa_1 \cdots \sa_s$)}
\end{array}$$  
In particular, since $(p,q)=1$, we obtain
$$|C_N(x)|_p= |C_{G(de,1,r)}(g_{\ueta})|_p=|C_{G(de/q,1,r/q)}(g_{\ueta}^{(q)})| _p.$$
%\begin{equation}
%\label{eq:centralisateur}
%|C_N(x)|_p= |C_{G(de,1,r)}(g_{\ueta})|_p=|C_{G(de/q,1,r/q)}(g_{\ueta}^{(q)})| _p.
%\end{equation}
Similarly, we have
%\begin{equation}
%\label{eq:centralisateur'}
$|C_{N'}(x')|_p= |C_{G(de,1,r')}(g_{\ueta'})|_p=|C_{G(de/q,1,r'/q)}(g_{\ueta'}^{(q)})| _p$.
%\end{equation}

\noindent
From these, and from (\ref{eq:Ichapeau}), we obtain
$$e \dis \frac{\widehat{I}(x,x')}{|C_{N}(x)|_p} =\dis \sum_{q | e \atop{g_{\ueta}, g_{\ueta'} \atop \mbox{\tiny{$q$-good}}}} q^s q^{s'}\sum_{0 \leq k< e \atop \mbox{\tiny{ord}}_e(k)=q} \overline{\e^k(g)} \e^{k'}(g') \frac{ \widehat{J}_{q,\uga,\uga',\uw} (g_{\ueta}^{(q)},g_{\ueta'}^{(q)})}{|C_{G(de/q,1,r/q)}(g_{\ueta}^{(q)})| _p}.$$
Since, for all $q$ dividing $e$, $J_{q,\uga,\uga',\uw} $ is a perfect isometry (by \cite[Theorem 5.4]{BrGr3}), we have $ \widehat{J}_{q,\uga,\uga',\uw} (g_{\ueta}^{(q)},g_{\ueta'}^{(q)}) \in |C_{G(de/q,1,r/q)}(g_{\ueta}^{(q)})| _p  {\cal R}$. Furthermore, the ring ${\cal R}$ contains the integers $q^s q^{s'}$, and the roots of unity $\overline{\e^k(g)} \e^{k'}(g')$. Hence $e \widehat{I}(x,x') \in |C_{N}(x)|_p {\cal R}$. Finally, since $(p,e)=1$, we obtain $\widehat{I}(x,x') \in |C_{N}(x)|_p {\cal R}$, as claimed.

\noindent
A similar argument shows that $\widehat{I}(x,x') \in |C_{N'}(x')|_p {\cal R}$, whence Property (1) of Definition \ref{def:perfectisometry} holds. This concludes the proof.

\end{proof}

\bigskip

{\bf Acknowledgements.} 
% The authors are grateful to M. Brou\'e for asking the question this article settles, at the end of a talk given by the first author at the Beijing Center for Mathematical Research during the Third International Symposium on Groups, Algebras and Related Topics, celebrating the 50th anniversary of the Journal of Algebra. 
Part of this work was done at the CIRM in Luminy 
during a research in pairs stay. The authors wish to thank the CIRM gratefully
for their financial and logistical support. The first author is
supported by Agence Nationale de la Recherche Projet GeRepMod
ANR-16-CE40-00010-01.
The second author also acknowledges financial support from the
Engineering and Physical Sciences Research Council grant \emph{Combinatorial Representation Theory} EP/M019292/1. 

 \bibliographystyle{abbrv}
\bibliography{references}

\end{document}